\documentclass[review]{elsarticle}

\usepackage{lineno,hyperref}
\modulolinenumbers[5]
\usepackage{leftidx}
\usepackage{mathtools}
 \usepackage{mathrsfs}
\usepackage{amssymb}
 \usepackage{amsmath}
 \usepackage[mathscr]{euscript}
 \usepackage{amsfonts}
 \usepackage{float}
 \usepackage{multirow}
 \usepackage{color}
 \usepackage{algorithmic,algorithm}
 \usepackage{amsthm}
 \usepackage{bm}

\journal{Journal of CMAME Templates}
\def\af{\alpha}
\def\p{\partial}

\def\md{\mathcal{D}}

\newcommand{\zd}{\,\mathrm{d}}
\newtheorem{theorem}{Theorem}[section]
\newtheorem{lemma}[theorem]{Lemma}
\theoremstyle{definition}
\newtheorem{definition}[theorem]{Definition}
\theoremstyle{remark}

\theoremstyle{example}
\newtheorem{example}[theorem]{Example}

 \usepackage{leftidx}

\def\vp{\varphi}

\begin{document}

\begin{frontmatter}

\title{Adaptive Finite Element Method for Fractional Differential Equations using Hierarchical Matrices\tnoteref{mytitlenote}}
\tnotetext[mytitlenote]{This work was supported by the OSD/ARO/MURI on
"Fractional PDEs for Conservation Laws and Beyond: Theory, Numerics and Applications
(W911NF-15-1-0562)" .}



\author[mymainaddress,mysecondaryaddress]{Xuan Zhao\corref{mycorrespondingauthor}}
\cortext[mycorrespondingauthor]{Corresponding author}
\ead{xuanzhao11@seu.edu.cn}

\author[tufts]{Xiaozhe Hu}
\ead{xiaozhe.hu@tufts.edu}

\author[mymainaddress,uncc]{Wei Cai}
\ead{wcai@csrc.ac.cn}

\author[Brown]{George Em Karniadakis}
\ead{george\_karniadakis@brown.edu}

\address[mymainaddress]{Division of Algorithms, Beijing Computational Science Research Center, Beijing 100084, P. R. China}
\address[mysecondaryaddress]{School of Mathematics, Southeast University, Nanjing 210096, P. R. China}
\address[tufts]{Department of Mathematics, Tufts University, Medford, MA 02155 USA}
\address[uncc]{Department of Mathematics and Statistics, University of North Carolina at Charlotte, Charlotte, NC 28223-0001}
\address[Brown]{Division of Applied Mathematics, Brown University, Providence, RI 02912, USA}

\begin{abstract}
A robust and fast solver for the fractional differential equation (FDEs) involving the Riesz fractional derivative is developed using an adaptive finite element method on non-uniform meshes. It is based on the utilization of hierarchical matrices ($\mathcal{H}$-Matrices) for the representation of the stiffness matrix resulting from the finite element discretization of the FDEs. We employ a geometric multigrid method for the solution of the algebraic system of equations. We combine it with an adaptive algorithm based on a posteriori error estimation to deal with general-type singularities arising in the solution of the FDEs. Through various test examples we demonstrate the efficiency of the method and the high-accuracy of the numerical solution even in the presence of singularities. The proposed technique has been verified effectively through fundamental examples including Riesz, Left/Right Riemann-Liouville fractional derivative and, furthermore, it can be readily extended to more general fractional differential equations with different boundary conditions and low-order terms. To the best of our knowledge, there are currently no other methods for FDEs that resolve singularities accurately at linear complexity as the one we propose here.
\end{abstract}

\begin{keyword}
Riesz fractional derivative, Hierarchical Matrices, geometric multigrid method, adaptivity, non-smooth solutions, finite element method
\end{keyword}

\end{frontmatter}


\section{Introduction}
Numerical methods for differential equations involving fractional derivatives either in time or space have been studied widely since they have various new scientific applications \cite{Re+2014+PRL,Sabatelli+2002+EP,Shlesinger+1987+PRL,Sun+2009+PA,Mainardi+2014,Samko+1993,Podlubny+1999}. However, as the size of the problem increases, the time required to solve the final system of equations increases considerably due to the nonlocality of the fractional differential operators  \cite{Yang+2010+AMM,Muslih+2010+JTP,Roop+2006+NMPDE,Sugimoto+1991+JFM,Roop+2006+JCAM,Deng+2008+SINUM,Bu+2014+JCP,Tian+2015+MC,Wang+2012+SISC,Wang+2011+JCP,Pang+2015+JCP}. Generally, there are two main difficulties in solving space fractional problems. First, the discretization matrix obtained by the utilization of the numerical methods is fully populated. This leads to increased storage memory requirements as well as increased solution time. Since the matrix is dense, the memory required to store its coefficients is of order $O(N^2)$, where $N$ denotes the number of unknowns, and the solution of the system requires $O(N^3)$ operations if direct solvers are used. Second, the solution of a space fractional problem has singularities around the boundaries even with smooth input data since the definition involves the integration of weak singular kernels.

So far good progress has been made on reducing the computational complexity  of the space FDEs for uniform mesh discretizations based on the 
observation that the coefficient matrices are Toeplitz-like if a {\em uniform} mesh is employed. This results in  efficient matrix-vector multiplications  \cite{Lin+2014+JCP,Pang+2012+JCP,Wang+2010+JCP,Wang+2013+JCP,jia2016fast}, which in conjunction with 
effective preconditioners lead to enhanced efficiency \cite{Lei+2013+JCP,WangDu+2013+JCP,Jia+2015+JCP,Moroney+2013+CMA,Pan+2014+SISC}. 
The multigrid method \cite{Pang+2012+JCP,Zhou+2013+JAM,Chen+2014+BIT,Xu+2015+MG} has also been employed to reduce the computational cost to 
$O(N\log (N))$.  However, to the best of our knowledge, most of the existing fast solvers depend on the Toeplitz-like structure of the matrices, which implies that the underlying meshes have to be uniform.  Therefore, the existing efficient solvers cannot be readily applied when a nonuniform mesh is used, including the important case of non-uniform (geometric) meshes generated by adaptive discretizations employed to deal with boundary singularities. 
Another effective approach in dealing with such singularities is by tuning the appropriate basis, e.g.. in Galerkin and collocation spectral methods. For example, the weighted Jacobi polynomials can be used to accommodate the weak singularity  if one has some information about the solution \cite{Zayernouri+2014+SISC,Zeng+2015+SISC,Chen+2014+ar,Zhao+2015+ar}; the proper choice of the basis results in significant improvement in the accuracy of numerical solutions. However, in the current work we assume that we do not have any information on the non-smoothness of the solution.

Some authors considered the multi-term space fractional differential equations and the factional initial value problem involving Rieman-Liouville fractional derivatives \cite{Chen+2014+ar,JiangJMAA2012,Zayernouri+2014+SISC}, and discussed the singular property of the solutions. 
As a general example, we consider the factional
initial value problem of order $1<\af<2$ involving the Riesz fractional derivative
\begin{align}
\md^\af_xu(x)=f(x),\,x\in(b,c) ,\label{Pro1}
\end{align}
where $ f \in L^2([b,c])$

$\mathcal{H}$-Matrices \cite{Hackbusch+1999+C,lin2011fast} have been developed over the last twenty years as a powerful data-sparse approximation of dense matrices. 
This representation has been used for solving integral equations and elliptic partial differential equations \cite{Bebendorf+2000+NM,Borm+2005+NM,Ho+2015+CPAM,schmitz2014fast}. The main advantage of $\mathcal{H}$-Matrices is the reduction of storage requirement, e.g. when storing a dense matrix, which requires $O(N^2)$ units of storage, while $\mathcal{H}$-Matrices provide an approximation requiring only $O(N k\,\log(N))$ units of storage, where $k$ is a parameter controlling the accuracy of the approximation.

In this work, instead of using the Topelitz-like structure of the matrices to reduce the computational complexity, we adapt the $\mathcal{H}$-Matrices representation to approximate the dense matrices arising from the discretization of the FDEs.  Our $\mathcal{H}$-Matrices approach does not restrict to the uniform meshes and can be easily generalized to the non-uniform meshes Therefore, it is suitable for the adaptive finite element method (AFEM) for FDEs.  We will show theoretically that the error of such $\mathcal{H}$-Matrices representation decays like $\mathcal{O}(3^{-k})$ while the storage complexity is $O(N k\,\log(N))$.  Moreover, in order to solve the linear system involving $\mathcal{H}$-Matrices efficiently, we develop a geometric multigrid (GMG) method based on the $\mathcal{H}$-Matrices representations and the resulting GMG method converges uniformly, which implies that the overall computational complexity for solving the linear system is $O(N k\,\log(N))$.  Since our $\mathcal{H}$-Matrices and GMG methods can be applied to non-uniform meshes, we also designed an adaptive finite element method (AFEM) for solving the FDEs. Similar to the standard AFEM for integer-order partial differential equations, our AFEM algorithm involves four main modules: \texttt{SOLVE}, \texttt{ESTIMATE}, \texttt{MARK}, and \texttt{REFINE}.  Here, the $\mathcal{H}$-Matrices approach and GMG method are used in the \texttt{SOLVE} module to reduce the computational cost.  An a posteriori error estimator based on gradient recovery approach is applied in the \texttt{ESTIMATE} module.  Standard D\"{o}flers marking strategy and bisection refinement are employed in the \texttt{MARK} and \texttt{REFINE}, respectively.  Thanks to the newly designed algorithm for solving the linear system of equations, the new AFEM for FDEs achieves the optimal computational complexity, while it also obtains the optimal convergence order.  The key to such AFEM algorithm for FDEs is the optimal linear solver we developed based on $\mathcal{H}$-Matrices representation and the GMG method.

The remainder of the paper is structured as follows.  In Section \ref{sec:preliminaries}, we introduce the  FDE considered in this
work.  Its finite element discretization and $\mathcal{H}$-Matrices representation are discussed in Section \ref{sec:H-mat}.  In Section \ref{sec:MG-Hmat}, we introduce the GMG method based on the $\mathcal{H}$-Matrices representation.  The overall AFEM algorithm is discussed in detail in Section \ref{sec:AFEM-FPDEs}, and numerical experiments are presented in Section \ref{sec:numerics} to demonstrate the high efficiency and accuracy of the proposed new method.

\section{Preliminaries} \label{sec:preliminaries}
In this section, we present some notations and lemmas which will be used in the following sections.

\begin{definition}
The fractional integral  of order  $\af,$ which is a complex number in the half-plane $Re(\af) > 0$, for the function $f(x)$ is  defined as
\begin{align*}
(\prescript{}{b}{\mathcal{I}}_{x}^{\af}f)(x)=\frac1{\Gamma(\af)}\int_{b}^x\frac{f(s)}{(x-s)^{1-\af}}\zd s,\quad x>b.
\end{align*}
\end{definition}

\begin{definition}\label{def_c}
The Caputo fractional derivative of order $\af\in(1,2)$ for the function $f(x)$ is  defined as
\begin{align*}
(\prescript{C}{b}{\mathcal{D}}_{x}^{\af}f)(x)=\prescript{}{b}{\mathcal{I}}_{x}^{2-\af}\left[\frac{\zd^2 }{\zd x^2}f(x)\right]
=\frac1{\Gamma(2-\af)}\int_{b}^x\frac{f^{\prime\prime}(s)}{(x-s)^{\af-1}}\zd s,\quad x>b.
\end{align*}
\end{definition}

\begin{definition}
The Left Riemann-Liouville fractional derivative of order  $\af\in(1,2)$ for the function $f(x)$ is  defined as
\begin{align*}
(\prescript{RL}{b}{\mathcal{D}}_{x}^{\af}f)(x)=\frac{\zd^2 }{\zd x^2}\left[(\prescript{}{b}{\mathcal{I}}_{x}^{2-\af}f)(x)\right]=\frac1{\Gamma(2-\af)}\frac{\zd^2 }{\zd x^2}\int_{b}^x\frac{f(s)}{(x-s)^{\af-1}}\zd s,\quad x>b.
\end{align*}
\end{definition}

\begin{definition}
The Right Riemann-Liouville fractional derivative of order  $\af\in(1,2)$ for the function $f(x)$ is  defined as
\begin{align*}
(\prescript{RL}{x}{\mathcal{D}}_{c}^{\af}f)(x)=\frac{1}{\Gamma(1-\af)}\left(-\frac{\zd^2 }{\zd x^2}\right)\int_{x}^{c}\frac{f(s)}{(s-x)^{\af-1}}\zd s,\quad x<c.
\end{align*}
\end{definition}

\begin{definition}\cite{Samko+1993} The Riesz fractional derivative of order  $\af\in(1,2)$ for the function $f(x)$ is  defined as
\begin{align*}
\md^\af_xf(x)&=-\frac1{2\cos(\af\pi/2)\Gamma(2-\af)}\frac{\zd^2}{\zd
x^2}\int_{b}^{c}|x-\xi|^{1-\af}f(\xi)\zd\xi\\
&=-\frac1{2\cos(\af\pi/2)}\left[ \prescript{RL}{b}{\mathcal{D}}_{x}^{\alpha}f(x)+\prescript{RL}{x}{\mathcal{D}}_{c}^{\alpha}f(x)\right].
\end{align*}
\end{definition}

\section{Discretization of the problem based on $\mathcal{H}$-Matrices representation} \label{sec:H-mat}
We consider the general model Eq. \eqref{Pro1} subject to the boundary conditions $u(b)=0,\;u(c)=0.$  Following the Galerkin approach, we solve equation \eqref{Pro1} projected onto the finite dimensional space $\mathcal{V}:=\text{span}\{\vp_1,\cdots,\vp_{N}\}$ and $\mathcal{V} \subset H^1_0([b,c])$, where $H^1_0([b,c])$ is the standard Sobolev space on $[b,c]$ and $\{\vp_i\}$ are standard piecewise linear basis functions defined on a mesh $b = x_0 < x_1 < \cdots < x_N < x_{N+1} = c$ with meshsize $h_i = x_{i+1} - x_i$, $i = 1,2, \cdots, N$.  We multiply $v \in \mathcal{V}$ by \eqref{Pro1} and integrate over $[b,c]$, 
\begin{align}
\int_{b}^{c}\md^\af_xu(x)\vp_i(x)\zd x=\int_{b}^{c} f(x)\vp_i(x)\zd x.
\end{align}
By integration by parts, we obtain the following weak formulation of \eqref{Pro1}: find $u(x) \in \mathcal{V}$, such that
\begin{align*}
\int_{b}^{c}  \left[\frac1{2 c(\af)}\frac{\zd}{\zd
x}\int_{b}^{c}|x-\xi|^{1-\af}u(\xi)\zd\xi  \right]    v^\prime(x)\zd x=\int_{b}^{c}f(x) v(x)\zd x, \quad \forall v \in \mathcal{V}
\end{align*}
where $c(\af)=\cos(\af\pi/2)\Gamma(2-\af).$

We rewrite the discrete solution $u_n=\sum_{j=1}^{N}u_j\vp_j \in \mathcal{V}$ and then the coefficient vector $\bm{u} = (u_1, u_2, \cdots, u_N)$ is the solution of the linear system
$$A  \bm{u}= \bm{f},$$
where 
\begin{align}
A _{ij}:=\int_{b}^{c}  \left[\frac1{2\cos(\af\pi/2)\Gamma(2-\af)}\frac{\zd}{\zd
x}\int_{b}^{c}|x-\xi|^{1-\af}\vp_j(\xi)\zd\xi  \right]    \vp^\prime_i(x)\zd x,
\end{align}
and 
\begin{align}
f_{i}:=\int_{b}^{c}  \vp_i(x)f(x)\zd x.
\end{align}

The matrix $A $ is dense as all entries are nonzero. Our aim is to approximate $A $ by a matrix $\tilde{A }$ which can be stored in a data-sparse (not necessarily sparse) format. The idea is to replace the kernel $\mathcal{S}(x,\xi)=|x-\xi|^{1-\alpha}$ by a degenerate kernel
\begin{align}
\tilde{\mathcal{S}}(x,\xi)=\sum_{\nu=0}^{k-1}p_\nu(x)q_\nu(\xi).
\end{align}

\subsection{Taylor Expansion of the Kernel}

Let $\tau:=[a^\prime,b^\prime]$, $\sigma:=[c^\prime,d^\prime]$, $\tau\times\sigma\subset[c,d]\times[c,d]$ be a subdomain with the property $b^\prime<c^\prime$ such that the intervals are disjoint: $\tau\cap\sigma=\emptyset.$ Then the kernel function is nonsingular in $\tau\times\sigma.$

\begin{lemma}[Derivative of left/right kernel]
\begin{align*}
\p_x^\nu\left[ (x-\xi)^{1-\alpha}\right]&=(-1)^{\nu}\prod_{l=1}^\nu (\af+l-2)(x-\xi)^{1-\af-\nu},\\
\p_x^\nu\left[ (\xi-x)^{1-\alpha}\right]&=\prod_{l=1}^\nu (\af+l-2)(\xi-x)^{1-\af-\nu}.
\end{align*}
\end{lemma}

Then we can use the truncated Talyor series at $x_0:=(a^\prime+b^\prime)/2$ to approximate the kernel and eventually obtain an approximation of the stiffness matrix.

\begin{align}
\tilde{\mathcal{S}}(x,\xi)&:=\sum_{\nu=0}^{k-1}\frac{1}{\nu!}\left[\prod_{l=1}^\nu (\af+l-2)(\xi-x_0)^{1-\af-\nu}\right](x-x_0)^\nu\label{KerTay}\\
&:=\sum_{\nu=0}^{k-1}p_\nu(x)q_\nu(\xi),
\end{align}
where 
\begin{align}
p_\nu(x)&=(x-x_0)^\nu,\\
q_\nu(\xi)& =\frac{1}{\nu!}\prod_{l=1}^\nu (\af+l-2)(\xi-x_0)^{1-\af-\nu} .
\end{align}
\subsection{ Low rank approximation of Matrix Blocks}
On certain subdomains (the condition will be made explicitly later), we can approximate the kernel $\mathcal{S}$ by the truncated Taylor series $\tilde{\mathcal{S}}$ from \eqref{KerTay} and replace the matrix entries $A _{ij}$ by the use of the degenerate kernel $\tilde{\mathcal{S}}(x,\xi)$ for the indices $(i,j)\in t\times s:$
\begin{align} \label{def:tildeA_ij}
\tilde{A }_{ij}&=\int_{b}^{c}  \left[\frac1{2\cos(\af\pi/2)\Gamma(2-\af)}\frac{\zd}{\zd
x}\int_{b}^{c}  \tilde{\mathcal{S}}(x,\xi)   \vp_j(\xi)\zd\xi  \right]    \vp^\prime_i(x)\zd x,
\end{align}
in which the double integral is separated into two single integrals:
\begin{align}
\tilde{A }_{ij}&=\int_{b}^{c}  \left[\frac1{2\cos(\af\pi/2)\Gamma(2-\af)}\frac{\zd}{\zd
x}\int_{b}^{c} \sum_{\nu=0}^{k-1}p_\nu(x)q_\nu(\xi) \vp_j(\xi)\zd\xi  \right]    \vp^\prime_i(x)\zd x\\
&=\frac1{2\cos(\af\pi/2)\Gamma(2-\af)}\sum_{\nu=0}^{k-1}\left[ \int_{b}^{c} p^\prime_\nu(x) \vp^\prime_i(x)\zd x \right] \left[  \int_{b}^{c} q_\nu(\xi) \vp_j(\xi)\zd\xi   \right]  
\end{align}

Thus, the submatrix $A |_{t\times s}$ can be represented in a factorized form

$$A |_{t\times s}=  \frac{1}{2\cos(\af\pi/2)\Gamma(2-\af)} CR^T,\quad C\in \mathbb{R}^{t\times\{0,\cdots,k-1\}},\quad R\in \mathbb{R}^{s\times\{0,\cdots,k-1\}}$$
where the entries of the matrix factors $C$ and $R$ are
\begin{align}
C_{i\nu}:= \int_{b}^{c} p^\prime_\nu(x) \vp^\prime_i(x)\zd x ,\quad R_{j\nu}:=\int_{b}^{c} q_\nu(\xi) \vp_j(\xi)\zd\xi  .
\end{align}

\subsection{$\mathcal{H}$-Matrix Representation Error Estimate}
Now we estimate the error of the $\mathcal{H}$-Matrix representation.  Here we consider the case $b' < c' (i<j)$ and the case $d' < a' (i>j)$ follows exactly the same procedure.  We first need to rewrite $C_{i\nu}$ and $R_{j\nu}$ using Taylor expansions.  For $b' < c' (i<j)$, using the following Taylor expansions with $h_i$ the length of the element $(x_i, x_{i+1})$, we have
\begin{align*}
(x_{i-1} - x_0)^{\nu} &= (x_i - x_0)^{\nu} + \nu (x_i - x_0)^{\nu-1} (-h_i)\nonumber\\
&\quad+ \frac{\nu(\nu-1)}{2!} (\xi_i - x_0)^{\nu-2} (-h_i)^2, \quad \xi_i \in [x_{i-1}, x_i], \\
(x_{i+1} - x_0)^{\nu} &= (x_i - x_0)^{\nu} + \nu (x_i - x_0)^{\nu-1} (h_{i+1}) \nonumber\\
&\quad+ \frac{\nu(\nu-1)}{2!} (\xi_{i+1} - x_0)^{\nu-2} (h_{i+1})^2, \quad \xi_{i+1} \in [x_i, x_{i+1}],
\end{align*}
we have, for $\nu \geq 2$,
\begin{equation}\label{eqn:Ainu-taylor}
C_{i \nu} = -\frac{\nu(\nu-1)}{2!}\left [(\xi_i - x_0)^{\nu-2} h_i + (\xi_{i+1} - x_0)^{\nu-2} h_{i+1}\right].
\end{equation}
Similarly, using the following Taylor expansions
\begin{align*}
&(x_{j-1} - x_0)^{3-\alpha-\nu}  = (x_j - x_0)^{3-\alpha-\nu} + (3 - \alpha - \nu) (x_j - x_0)^{2 - \alpha - \nu} (-h_j) \\ 
& \quad + \frac{(3-\alpha-\nu)(2-\alpha -\nu)}{2!} (\xi_j - x_0)^{1 - \alpha -\nu} (- h_j)^2, \quad \xi_j \in [x_{j-1}, x_j], \\ 
&(x_{j+1} - x_0)^{3-\alpha-\nu} = (x_j - x_0)^{3-\alpha-\nu} + (3 - \alpha - \nu) (x_j - x_0)^{2 - \alpha - \nu} (h_{j+1}) \\
& \quad + \frac{(3-\alpha-\nu)(2-\alpha -\nu)}{2!} (\xi_{j+1} - x_0)^{1 - \alpha -\nu} (h_{j+1})^2, \quad \xi_{j+1} \in [x_{j-1}, x_j],
\end{align*}
we have, 
\begin{equation}\label{eqn:Bjnu-taylor}
R_{j \nu} = \frac{1}{2! \nu!} \Pi_{l=1}^{\nu} (\alpha + l - 2) \left[  (\xi_j - x_0)^{1-\alpha -\nu} h_j + (\xi_{j+1} -x_0 )^{1-\alpha - \nu}  h_{j+1}  \right].
\end{equation}
Based on \eqref{eqn:Ainu-taylor} and \eqref{eqn:Bjnu-taylor}, we can analyze the element-wise error  in the following theorem.  

\begin{theorem}[Element-wise Approximation Error] \label{thm:element-wise-error}
Let $\tau := [a^\prime,b^\prime]$, $\sigma := [c^\prime,d^\prime]$, $b' < c'$, $x_0 = (a' + b')/2$, and $k\geq 2$, we have
\begin{align} \label{ine:element-wise-error}
&\quad|A _{ij} - \tilde{A }_{ij} | \nonumber \\
& \leq \frac{1}{8 c(\af)}  \frac{k(k-1)(h_i + h_{i+1})(h_j + h_{j+1})}{\left( |x_0 - a'| + |c' - b'| \right)^{\alpha-1}|x_0 -a'|^2} \left[ \frac{\rm{diam}(\tau) + 2\rm{dist}(\tau, \sigma) }{2\rm{dist}(\tau, \sigma) } \right]^3 \left[  1+  2 \frac{\rm{dist}(\tau, \sigma)}{\rm{diam}(\tau)} \right]^{-k}.
\end{align}
where $\rm{diam}(\tau)$ is the diameter of $\tau$ and $\rm{dist}(\tau, \sigma)$ is the distance between the intervals $\tau$ and $\sigma$.
If $ \frac{\rm{dist}(\tau, \sigma)}{\rm{diam}(\tau)} \geq 1 $, we have
\begin{align} \label{ine:element-wise-error-1}
 |A _{ij} - \tilde{A }_{ij} | 
\leq \frac{27}{64} \frac{1}{ c(\af)}  \frac{k(k-1)(h_i + h_{i+1})(h_j + h_{j+1})}{\left( |x_0 - a'| + |c' - b'| \right)^{\alpha-1} |x_0 - a'|^2} \left( 3 \right)^{-k}.
\end{align}
\end{theorem}
\begin{proof}
Let us consider the case $b' < c' (i<j)$, according to \eqref{eqn:Ainu-taylor} and \eqref{eqn:Bjnu-taylor}, for $\nu \geq 2$, we have
\begin{align*}
|C_{i\nu}| &\leq  \frac{\nu (\nu -1)}{2} (h_i + h_{i+1}) |x_0 - a'|^{\nu -2}, \\
|R_{j\nu}| & \leq \frac{1}{2 \nu!} \left[ \Pi_{l=1}^\nu (\alpha + l -2)  \right] (h_j + h_{j+1}) \left(  \frac{1}{|x_0 - a'| + |c' - b'|}  \right)^{\nu + \alpha - 1}.
\end{align*}
Denote $r:= \frac{|x_0 - a'|}{|x_0 - a' | + |c' - b'|} <1$ and use the fact that $\frac{\Pi_{l=1}^\nu (\alpha + l -2) }{\nu!} \leq 1$, we have
\begin{align*}
|C_{i \nu} R_{j\nu}| \leq \frac{1}{4} \frac{(h_i + h_{i+1})(h_j + h_{j+1})}{\left( |x_0 - a'| + |c' - b'| \right)^{\alpha+1}} \left[  \nu (\nu-1) r^{\nu - 2} \right].
\end{align*}
Therefore, for $k \geq 2$
\begin{align*}
|A _{ij} - \tilde{A }_{ij} | & \leq \frac{1}{2 c(\af)} | \sum_{\nu = k}^{\infty} C_{i\nu}R_{j\nu}| \leq \frac{1}{2 c(\af)}  \sum_{\nu = k}^{\infty} | C_{i\nu}R_{j\nu}|  \\
& \leq \frac{1}{8 c(\af)}  \frac{(h_i + h_{i+1})(h_j + h_{j+1})}{\left( |x_0 - a'| + |c' - b'| \right)^{\alpha+1}}  \left[\sum_{\nu=k}^{\infty} \nu(\nu-1) r^{\nu-2} \right] \\
& = \frac{1}{8 c(\af)}  \frac{(h_i + h_{i+1})(h_j + h_{j+1})}{\left( |x_0 - a'| + |c' - b'| \right)^{\alpha+1}}  \frac{(k-1)(k-2)r^2 - 2k(k-2)r + k(k-1)}{(1-r)^3} r^{k-2} \\
& \leq \frac{1}{8 c(\af)}  \frac{(h_i + h_{i+1})(h_j + h_{j+1})}{\left( |x_0 - a'| + |c' - b'| \right)^{\alpha+1}} \frac{k(k-1)}{(1-r)^3} r^{k-2} \\
&= \frac{1}{8 c(\af)}  \frac{(h_i + h_{i+1})(h_j + h_{j+1})}{\left( |x_0 - a'| + |c' - b'| \right)^{\alpha-1} | x_0 - a' |^2} \frac{k(k-1)}{(1-r)^3} r^{k} 
\end{align*}
Note that $r  = \frac{\rm{diam}(\tau)}{\rm{diam}(\tau) + 2\rm{dist}(\tau, \sigma)}$ and $k \geq 2$, we have
\begin{align*}
&|A _{ij} - \tilde{A }_{ij} | \\
& \leq \frac{1}{8 c(\af)}  \frac{k(k-1)(h_i + h_{i+1})(h_j + h_{j+1})}{\left( |x_0 - a'| + |c' - b'| \right)^{\alpha-1} |x_0 - a'|^2} \left[ \frac{\rm{diam}(\tau) + 2\rm{dist}(\tau, \sigma) }{2\rm{dist}(\tau, \sigma) } \right]^3 \left[  1+  2 \frac{\rm{dist}(\tau, \sigma)}{\rm{diam}(\tau)} \right]^{-k},
\end{align*}
which gives \eqref{ine:element-wise-error}.  

If $ \frac{\rm{dist}(\tau, \sigma)}{\rm{diam}(\tau)} \geq 1 $, we have $\frac{\rm{diam}(\tau) + 2\rm{dist}(\tau, \sigma) }{2\rm{dist}(\tau, \sigma) }  \leq \frac{3}{2}$, $ 1+  2 \frac{\rm{dist}(\tau, \sigma)}{\rm{diam}(\tau)} \geq 3$, and then
\begin{align*}
|A _{ij} - \tilde{A }_{ij} | & \leq \frac{27}{64} \frac{1}{ c(\af)}  \frac{k(k-1)(h_i + h_{i+1})(h_j + h_{j+1})}{\left( |x_0 - a'| + |c' - b'| \right)^{\alpha-1} |x_0 - a'|^2} \left( 3 \right)^{-k},
\end{align*}
which completes the proof. 
\end{proof}

In the error estimate \eqref{ine:element-wise-error}, we can see that the dominating term is $\left[  1+  2 \frac{\rm{dist}(\tau, \sigma)}{\rm{diam}(\tau)} \right]^{-k}$, which determines the decaying rate and, thus, the quality of the approximation.  If $\rm{dist}(\tau, \sigma) \rightarrow 0$, the approximation will degenerate. However, if we require $\rm{diam}(\tau) \leq \rm{dist}(\tau, \sigma)$ as in the Theorem \ref{thm:element-wise-error}, we can have a nearly uniform bound 
\begin{equation*}
|A _{ij} - \tilde{A }_{ij} |  = \mathcal{O}(3^{-k}),
\end{equation*}
where $c$ depends on the intervals and $\alpha$ weakly since it is dominated by $3^{-k}$.  Moreover, the element-wise error mainly depends on the ratio $\frac{\rm{dist}(\tau, \sigma)}{\rm{diam}(\tau)}$ if we assume $\rm{diam}(\tau) \leq \rm{dist}(\tau, \sigma)$, and the bigger the ratio the better the approximation. This is equivalent to stating that the bigger the distance $\rm{dist}(\tau, \sigma)$ compared to $\rm{diam}(\tau)$ is, the faster the approximation error decays. Therefore, we call $\tau \times \sigma$ \emph{admissible} if $\rm{diam}(\tau) \leq \rm{dist}(\tau, \sigma)$, and this is the condition that is required to be able to use low rank approximation on the subdomains. The error decays exponentially with respect to the order $k$.  Figure \ref{fig:element-wise-error} shows how the error decays with respect to k when we vary the ratio or $\alpha$; we can see clearly that the error depends strongly on the ratio rather than $\alpha$, which supports our error estimates.

\begin{figure}[ht!]
\caption{Error in the Frobenius norm ($\| A - \tilde{A} \|_{F}$, where $\| M \|_F:= \sqrt{ \sum_{i,j} |M_{ij}|^2 } $) mainly depends on the ratio rather than on $\alpha$. Left: $\text{ration} := \frac{\rm{dist}(\tau, \sigma)}{\rm{diam}(\tau)}$ with $\rm{diam}(\tau)$ the diameter of the interval $\tau$ and $\rm{dist}(\tau, \sigma)$ the distance between the intervals $\tau$ and $\sigma$; Right: $\alpha \in (1,2)$ is the order of the fractional derivative.} 
\label{fig:element-wise-error}
\includegraphics[width=0.49\textwidth]{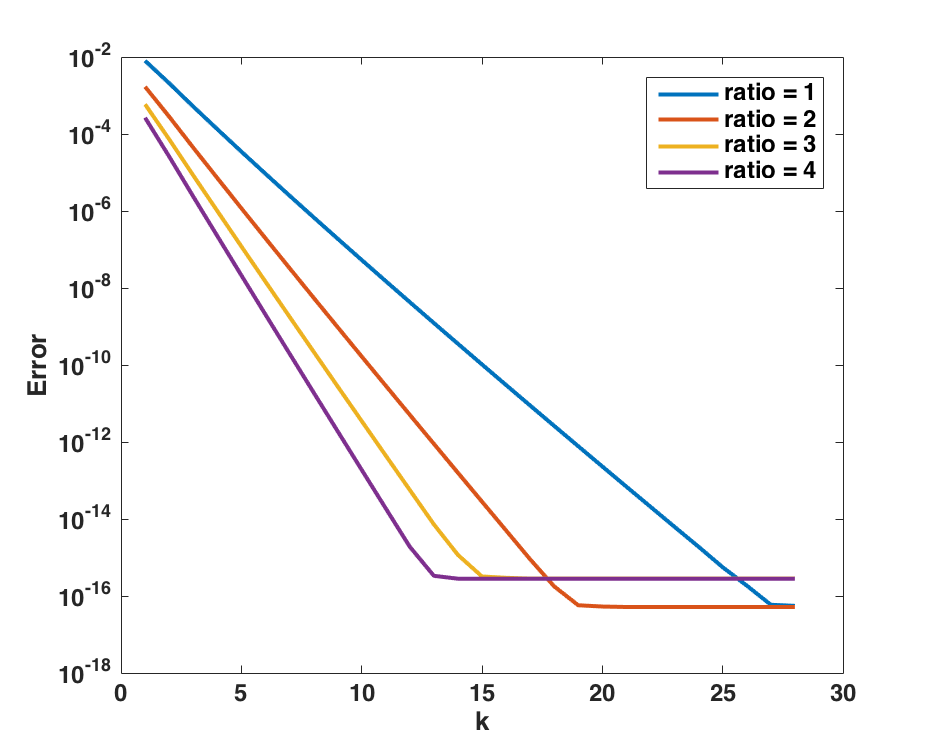} \quad
\includegraphics[width=0.49\textwidth]{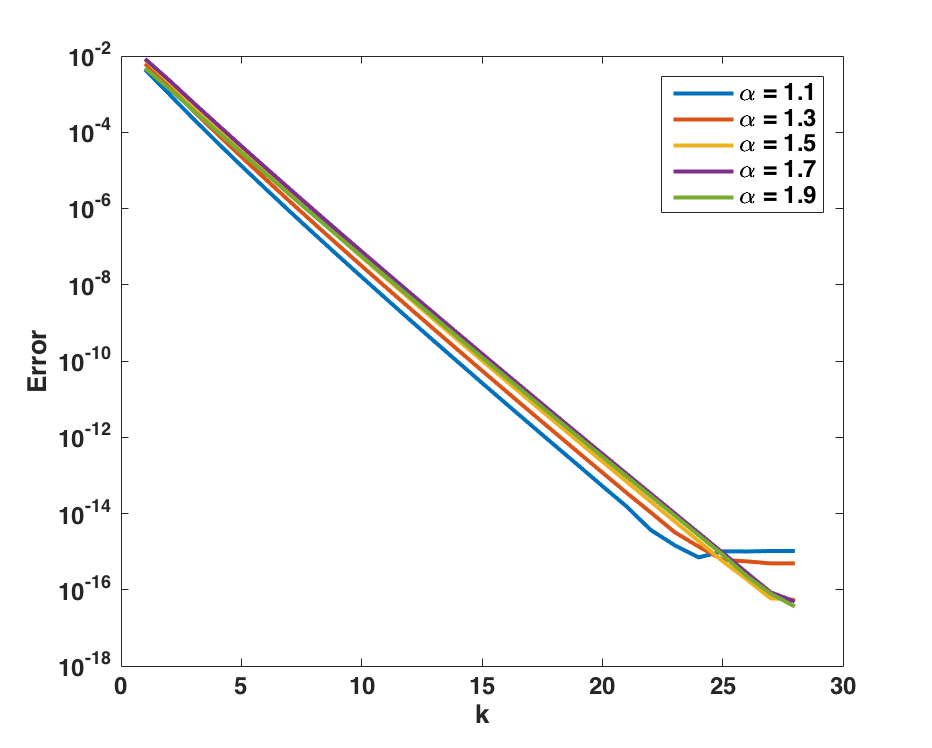}
\end{figure}

\section{Geometric Multigrid Method based on $\mathcal{H}$-Matrices} \label{sec:MG-Hmat}
In this section, we discuss how we solve the linear system of equations in the $\mathcal{H}$-Matrix format.  Although many existing fast linear solvers based on the $\mathcal{H}$-Matrix format can be applied directly, for example, Hierarchical inversion and $\mathcal{H}$-Matrix LU decomposition (see, e.g. \cite{Bebendorf.M2008a}), we will design the geometric multigrid (GMG) method based on the $\mathcal{H}$-matrix format and use GMG method to solve the linear system.  The reason is the following.  Firstly, since we are solving discrete systems resulting from differential equations, the GMG method is known to be one of the optimal methods and suitable for large-scale problems.  In \cite{Xu+2015+MG}, GMG methods have been introduced to FDEs discretized on uniform grids.  Therefore, it is natural to take advantage of the $\mathcal{H}$-matrices and generalize GMG methods for FDEs in higher dimensions discretized on non-uniform grids.  Secondly, our ultimate goal is to design adaptive finite element methods for solving FDEs, therefore, hierarchical grids are available due to the adaptive refinement procedure (which will be made clear in Section \ref{sec:AFEM-FPDEs}); hence, we can use those nested unstructured grids and design GMG methods accordingly.  Next, we will present the GMG algorithms.

As usual, the multigrid method is built upon the subspaces that are defined on nested sequences of triangulations.  We assume that we start with an initial grid $\mathcal{T}_0$ and a nested sequence of grids $\{ \mathcal{T}_\ell \}_{\ell=0}^J$, where $\mathcal{T}_\ell$ is obtained by certain refinement procedure of $\mathcal{T}_{\ell-1}$ for $\ell>0$, i.e., 
\begin{equation*}
\mathcal{T}_0 \leq \mathcal{T}_1 \leq \cdots \leq \mathcal{T}_J = \mathcal{T}.
\end{equation*}
Let $\mathcal{V}_\ell$ denote the corresponding linear finite element space based on $\mathcal{T}_\ell$.  We thus get a sequence of multilevel spaces 
\begin{equation*}
\mathcal{V}_0 \subset \mathcal{V}_1 \subset \cdots \subset \mathcal{V}_J = \mathcal{V}.
\end{equation*}
Note that, a natural space decomposition of $\mathcal{V}$ is $\mathcal{V} = \sum_{\ell=0}^J \mathcal{V}_\ell$ and this is not a direct sum.  Based on these finite element spaces, we have the following linear system of equations on each level:
\begin{equation} \label{eqn:linear-system}
A_\ell u_\ell = f_\ell, \quad \ell = 0, 1, \cdots, J.
\end{equation}
Correspondingly, we also have their $\mathcal{H}$-Matrix approximation on each level
\begin{equation} \label{eqn:H-mat-linear-system}
\tilde{A}_\ell \tilde{u}_\ell = f_\ell, \quad \ell= 0,1, \cdots, J,
\end{equation}
where $\tilde{A}_\ell$ is the $\mathcal{H}$-Matrix representation of $A$ as defined entry-wise by~\eqref{def:tildeA_ij}.

In practice, we will solve \eqref{eqn:H-mat-linear-system} based on the GMG method. Because $\tilde{A}_\ell$ provides a good approximation to $A_\ell$ on each level $\ell$, we can expect that $\tilde{u}_\ell$ provides a good approximation to $u_\ell$ on each level $\ell$ based on the standard perturbation theory of solving linear systems of equations \cite{Golub.G;Van-Loan.C1996}.  In order to define the GMG method, we need to introduce the standard prolongation $I_\ell$ on level $\ell$, which is the matrix representation of the standard inclusion operator from $\mathcal{V}_{\ell-1} = \text{span}\{  \varphi_1^{\ell-1}, \varphi_2^{\ell-1}, \cdots, \varphi_{N_{\ell-1}}^{\ell-1} \}$ to $\mathcal{V}_\ell = \text{span}\{  \varphi_1^{\ell}, \varphi_2^{\ell}, \cdots, \varphi_{N_{\ell}}^{\ell} \} $ since $\mathcal{V}_{\ell-1} \subset \mathcal{V}_\ell$, e.g., $I_\ell \in \mathbb{R}^{N_{\ell} \times N_{\ell-1}}$ such that,
\begin{equation} \label{def:prolongation}
(I_{\ell})_{ij} = \beta_{ij}, \ \text{where} \  \varphi_{j}^{\ell-1} = \sum_{i = 1}^{N_{\ell}} \beta_{ij} \varphi_{i}^{\ell}, \ j = 1, \cdots, N_{\ell-1}.
\end{equation}
Now we can define the standard $V$-cycle GMG method for solving \eqref{eqn:H-mat-linear-system} by the following recursive (Algorithm 1). 


\begin{algorithm}
\caption{V-cycle multigrid method for $\mathcal{H}$-Matrix $\tilde{A}_\ell$} \label{alg:V-cycle-H-mat}
$\tilde{u}_\ell = \texttt{Vcycle}(\tilde{u}_\ell, \tilde{A}_\ell, f_\ell, \ell)$
\begin{algorithmic}[1]
\IF {$\ell=0$}
\STATE $\tilde{u}_0 = \tilde{A}_0^{-1}f_0$
\ELSE 
\STATE $\tilde{u}_\ell = \texttt{FGSsmoother}(\tilde{u}_\ell, \tilde{A}_\ell, f_\ell)$
\STATE $r_\ell = f_\ell - \texttt{Hmatvec}(\tilde{A}_\ell, \tilde{u}_\ell)$
\STATE $r_{\ell-1} = I_\ell^T \bm{r}_\ell$
\STATE $e_{\ell-1} = 0$ and $e_{\ell-1} = \texttt{Vcycle}(e_{\ell-1}, H_{\ell-1}, r_{\ell-1}, \ell-1)$
\STATE $\tilde{u}_\ell = \tilde{u}_\ell + I_\ell e_{\ell-1}$
\STATE $\tilde{u}_\ell = \texttt{BGSsmoother}(\tilde{u}_\ell, \tilde{A}_\ell, f_\ell)$
\ENDIF
\end{algorithmic}
\end{algorithm}

We first want to point out that on the the coarsest level $\ell=0$, we need to solve the linear system exactly because the size of the problem is very small compared with the size on the finest level.  This involves inverting the $\mathcal{H}$-Matrix $\tilde{A}_0$ and can be done efficiently by Hierarchical inversion and $\mathcal{H}$-Matrix LU decomposition methods. However, in order to keep our implementation simple, we choose a very coarse grid $\mathcal{T}_0$ to start with and the size of the problem is small enough so that the computational cost can be ignored and, therefore, the $\mathcal{H}$-Matrix approach is not needed, i.e., $\tilde{A}_0 = A_0$ in our implementation.  Standard Gaussian Elimination or LU decomposition for a dense matrix is used to solve the linear system on the coarsest grid exactly.  

Secondly, Algorithm \ref{alg:V-cycle-H-mat} uses matrix-vector multiplication based on the $\mathcal{H}$-Matrix format, i.e. the subroutine $\texttt{Hmatvec}$.  Such a matrix-vector multiplication has been widely discussed in the $\mathcal{H}$-Matrix literature, for example \cite{Bebendorf.M2008a}.  We also adopt the standard implementation here.  We assume that matrix $H$ is stored using the $\mathcal{H}$-Matrix format and we want to compute $\bm{y} = H\bm{x}$.  The algorithm is presented in Algorithm \ref{alg:H-mat-vec}.

\begin{algorithm}
\caption{Matrix-vector multiplication in $\mathcal{H}$-Matrix format} \label{alg:H-mat-vec}
$\bm{y} = \texttt{Hmatvec}(H, \bm{x})$
\begin{algorithmic}[1]
\IF {$H$ is full matrix}
\STATE $\bm{y} = H \bm{x}$
\ENDIF
\IF {$H$ is low rank approximation, i.e. it is stored in factorized form $H = CR^T$}
\STATE $\bm{y} = C (R^T \bm{x})$
\ENDIF
\IF {$H$ is stored in $2$ by $2$ block form, i.e., $H = \begin{pmatrix} H_{11} & H_{12} \\ H_{21} & H_{22} \end{pmatrix} $}
\STATE partition $\bm{x}$ as $\bm{x} = \begin{pmatrix} \bm{x}_1 \\ \bm{x}_2 \end{pmatrix}$
\STATE $\bm{y}_1 = \texttt{Hmatvec}(H_{11}, \bm{x}_1) + \texttt{Hmatvec}(H_{12}, \bm{x}_2)$
\STATE $\bm{y}_2 = \texttt{Hmatvec}(H_{21}, \bm{x}_1) + \texttt{Hmatvec}(H_{22}, \bm{x}_2)$
\STATE set $\bm{y} = \begin{pmatrix} \bm{y}_1 \\ \bm{y}_2 \end{pmatrix}$
\ENDIF
\end{algorithmic}
\end{algorithm}

Finally, we discuss the smoothers used in the GMG Algorithm \ref{alg:V-cycle-H-mat}.  We use Gauss-Seidel smoothers in our implementation.  The Subroutine $\texttt{FGSsmoother}$ is the implementation of forward Gauss-Seidel method and the subroutine $\texttt{BGSsmoother}$ is the implementation of backward Gauss-Seidel method.  Their implementations are given by Algorithm \ref{alg:H-mat-FGS}  and \ref{alg:H-mat-BGS}, respectively.  

\begin{algorithm}
\caption{Forward Gauss-Seidel smoother in $\mathcal{H}$-Matrix format} \label{alg:H-mat-FGS}
$\bm{x} = \texttt{FGSsmoother}(H, \bm{b}, \bm{x})$
\begin{algorithmic}[1]
\STATE $\bm{r} = \bm{b} - \texttt{Hmatvec}(H, \bm{x})$
\IF {$H$ is full matrix}
\STATE get the lower triangular part $L$ of $H$
\STATE $\bm{z} =  L^{-1} \bm{r}$
\STATE $\bm{x} = \bm{x} + \bm{z}$
\ENDIF 
\IF {$H$ is stored in $2$ by $2$ block form, i.e., $H = \begin{pmatrix} H_{11} & H_{12} \\ H_{21} & H_{22} \end{pmatrix} $}
\STATE partition $\bm{r}$ as $\bm{r} = \begin{pmatrix} \bm{r}_1 \\ \bm{r}_2 \end{pmatrix}$
\STATE $\bm{z}_1 = \texttt{FGSsmoother}(H_{11}, \bm{r}_1, 0)$
\STATE $\bm{r}_2 = \bm{r}_2 - \texttt{Hmatvec}(H_{21}, \bm{z}_1)$
\STATE $\bm{z}_2 = \texttt{FGSsmoother}(H_{22}, \bm{r}_2, 0)$
\STATE set $\bm{z} = \begin{pmatrix} \bm{z}_1 \\ \bm{z}_2 \end{pmatrix}$
\STATE $\bm{x} = \bm{x} + \bm{z}$
\ENDIF
\end{algorithmic}
\end{algorithm}

\begin{algorithm}
\caption{Backward Gauss-Seidel smoother in $\mathcal{H}$-Matrix format} \label{alg:H-mat-BGS}
$\bm{x} = \texttt{BGSsmoother}(H, \bm{b}, \bm{x})$
\begin{algorithmic}[1]
\STATE $\bm{r} = \bm{b} - \texttt{Hmatvec}(H, \bm{x})$
\IF {$H$ is full matrix}
\STATE get the upper triangular part $U$ of $H$
\STATE $\bm{z} =  U^{-1} \bm{r}$
\STATE $\bm{x} = \bm{x} + \bm{z}$
\ENDIF 
\IF {$H$ is stored in $2$ by $2$ block form, i.e., $H = \begin{pmatrix} H_{11} & H_{12} \\ H_{21} & H_{22} \end{pmatrix} $}
\STATE partition $\bm{r}$ as $\bm{r} = \begin{pmatrix} \bm{r}_1 \\ \bm{r}_2 \end{pmatrix}$
\STATE $\bm{z}_2 = \texttt{BGSsmoother}(H_{22}, \bm{r}_2, 0)$
\STATE $\bm{r}_1 = \bm{r}_1 - \texttt{Hmatvec}(H_{12}, \bm{z}_2)$
\STATE $\bm{z}_1 = \texttt{BGSsmoother}(H_{11}, \bm{r}_1, 0)$
\STATE set $\bm{z} = \begin{pmatrix} \bm{z}_1 \\ \bm{z}_2 \end{pmatrix}$
\STATE $\bm{x} = \bm{x} + \bm{z}$
\ENDIF
\end{algorithmic}
\end{algorithm}

Note that in the Gauss-Seidel smoother algorithms, we do not consider the case that $H$ is a low rank approximation and is stored in factorized form $H = CR^T$.  This is because in the Gauss-Seidel algorithm, only the diagonal entries will be inverted and, in our $\mathcal{H}$-Matrix representation for the FDEs, the diagonal entries are definitely stored explicitly in the full matrix format.  Therefore, we consider the cases  where $H$ is stored in either full matrix format or $2$ by $2$ block format, where the diagonal entries need to be accessed recursively. 

Based on Algorithm \ref{alg:H-mat-vec}, \ref{alg:H-mat-FGS}, and \ref{alg:H-mat-BGS}, the V-cycle multigrid Algorithm \ref{alg:V-cycle-H-mat} is well-defined.  It is easy to check that the overall computational complexity of one V-cycle is $\mathcal{O}(k N_\ell \log N_\ell)$ where $N_\ell$ is the number of degrees of freedom on the level $\ell$ and $k$ is the upper bound of the rank used in the low rank approximation for the $\mathcal{H}$-Matrix.  This is because the computational cost of matrix-vector multiplication for $\mathcal{H}$-Matrix is $\mathcal{O}(k N_\ell \log N_\ell)$ and it is well-known that the Gauss-Seidel method has roughly the same computational cost as the matrix-vector multiplication.  In practice, $k \ll N_\ell$ usually is a small number and does not depends on $N_\ell$.  Therefore, we can say that the computational cost of V-cycle multigrid on level $\ell$ is roughly $\mathcal{O}(N_\ell \log N_\ell)$.  Note that the computational complexity for solving the linear system of equations in the $\mathcal{H}$-Matrix format, such as Hierarchical inversion and $\mathcal{H}$-Matrix LU decomposition methods, is $\mathcal{O}(k^2 N_\ell \log^2 N_\ell) \approx \mathcal{O}(N_\ell \log^2 N_\ell)$ \cite{Bebendorf.M2008a}.  Therefore, the GMG approach is slightly better in terms of computational cost, especially for large-scale problems.

\section{Adaptive Finite Element Method for Fractional PDEs} \label{sec:AFEM-FPDEs}
In this section, we discuss the adaptive finite element method (AFEM) for solving FDEs.  We follow the idea of standard AFEM, which is characterized by the following iteration
\begin{equation*}
\texttt{SOLVE} \quad  \longrightarrow \quad  \texttt{ESTIMATE} \quad  \longrightarrow \quad  \texttt{MARK}  \quad  \longrightarrow \quad  \texttt{REFINE}.
\end{equation*}
Such iteration generates a sequence of discrete solutions converging to the exact one.  We want to emphasize that one of the main difficulties of applying the AFEM to the FDEs is the \texttt{SOLVE} step.  The AFEM iteration usually generates non-uniform grids, which makes the resulting linear system difficult to solve by existing fast linear solvers. This is because usually the traditional approaches take advantage of the uniform grid and the Toeplitz structure of the resulting stiffness matrices, which is not true for the non-uniform grid.  However, in this paper, by introducing the $\mathcal{H}$-Matrix approach and GMG method, we can efficiently solve the linear system obtained on the non-uniform grid and, therefore, design an AFEM method for FDEs so that each AFEM iteration has computational complexity $\mathcal{O}(k N \log N)$, suitable for practice applications.  Next, we will introduce the four modules in the AFEM iterations. 

\subsection{\texttt{SOLVE} Module} 
As usual, the \texttt{SOLVE} module should take the current grid as the input and output the corresponding finite element approximation.  Note here that the current grid in general is obtained by adaptive refinement and, therefore, it is an unstructured grid.  So, our \texttt{SOLVE} module will use the hierarchical matrix representation mentioned in Section \ref{sec:H-mat} to assemble the linear system of equations stored in $\mathcal{H}$-Matrix format and solve it by the GMG method discussed in Section \ref{sec:MG-Hmat}.  The detailed description is listed below. 

\begin{algorithm}
\caption{\texttt{SOLVE} module: solve FDES using FEM} \label{alg:solve} 
$\tilde{u}_\ell = \texttt{SOLVE}(\mathcal{T}_\ell)$
\begin{algorithmic}[1]
\STATE On current grid $\mathcal{T}_\ell$, assemble the linear system of equation $\tilde{A}_\ell \tilde{u}_\ell = f_\ell$ and stored it in $\mathcal{H}$-Matrix format
\STATE Solve $\tilde{u}_\ell$ by V-cycle GMG method (Algorithm \ref{alg:V-cycle-H-mat}) directly or preconditioned Conjugate Gradient method with V-cycle GMG as a preconditioner
\end{algorithmic}
\end{algorithm}

Let $N_\ell$ denote the number of degrees of freedoms on grid $\mathcal{T}_\ell$.  As discussed in Section \ref{sec:H-mat}, assembling the $\mathcal{H}$-Matrix on grid $\mathcal{T}_\ell$ costs $\mathcal{O}(k N_\ell \log N_\ell)$ operations and solving $U_\ell$ by the GMG method also costs $\mathcal{O}(k N_\ell \log N_\ell)$ operations.  Therefore, the overall cost of the \texttt{SOLVE} module is $\mathcal{O}(k N_\ell \log N_\ell)$ or $\mathcal{O}(N_\ell \log N_\ell)$ when $k \ll N_\ell$.  

\subsection{\texttt{ESTIMATE} Module}
Given a grid $\mathcal{T}_\ell$ and finite element approximation $\tilde{u}_\ell \in \mathcal{V}_\ell$, the \texttt{ESTIMATE} module computes a posteriori error estimators $\{ \eta_\ell( \tilde{u}_\ell, \tau) \}_{\tau \in \mathcal{T}_\ell}$, which should be computable on each element $\tau \in \mathcal{T}_\ell$ and indicate the true error.  Such a posteriori error estimators have been widely discussed in the AFEM literature for second order elliptic PDEs discretized by FEM \cite{Cascon.J;Kreuzer.C;Nochetto.R;Siebert.K2008,Nochetto.R;Siebert.K;Veeser.A2009}.   There are three major error estimators: residual based error estimator, gradient recovery based error estimator, and objective-oriented error estimator.  For FDES, to the best of our knowledge, studies on such a posteriori error estimators are very limited.  In this work, we will adopt the \emph{gradient recovery} based error estimators due to its simplicity and problem-independence.  A more detailed investigation on such an error estimator will be a subject of our future work.  The detailed description of the \texttt{ESTIMATE} module is listed below.

\begin{algorithm}
\caption{\texttt{ESTIMATE} Module: gradient recovery type a posteriori error estimator} \label{alg:estimate}
$\{ \eta_\ell(\tilde{u}_\ell, \tau) \}_{\tau \in \mathcal{T}_\ell} = \texttt{ESTIMATE}(\tilde{u}_\ell, \mathcal{T}_\ell)$
\begin{algorithmic}[1]
\STATE Compute $\nabla \tilde{u}_\ell$.
\STATE Compute the recovery gradient $\mathcal{G} \tilde{u}_\ell \in \mathcal{V}_{\ell} = \text{span} \{ \varphi^{\ell}_1, \cdots, \varphi^\ell_{N_\ell} \}$ 
\begin{equation*}
\mathcal{G} \tilde{u}_\ell := \sum_{i}^{N_\ell} (\mathcal{G} \tilde{u}_\ell)_i \varphi_i^{\ell}, \quad (\mathcal{G} \tilde{u}_\ell)_i := \frac{h^\ell_i \nabla \tilde{u}_\ell |_{[x^{\ell}_{i-1},x^\ell_{i}]} + h^{\ell}_{i+1} \nabla \tilde{u}_\ell |_{[x^{\ell}_{i},x^\ell_{i+1}]} }{h^{\ell}_i + h^\ell_{i+1}},
\end{equation*}
where $x^{\ell}_i$ are the nodes in the grid $\mathcal{T}_{\ell}$ and $h^\ell_i = x^\ell_i - x^{\ell}_{i-1}$.
\STATE Compute the error estimator on each element $\tau \in \mathcal{T}_\ell$
\begin{equation*}
\eta_\ell(\tilde{u}_\ell, \tau) := \| \nabla \tilde{u}_\ell - \mathcal{G} \tilde{u}_\ell  \|_{\tau}, \quad \tau \in \mathcal{T}_\ell.
\end{equation*}
\end{algorithmic}
\end{algorithm}

It is easy to check that the computational complexity of \texttt{ESTIMATE} module is $\mathcal{O}(N_\ell)$ because all the computations are done locally on the elements $\tau$ and on each element $\tau$, the operations are finite and independent of $N_\ell$.

After computing the error estimators on each element $\tau$, the overall error estimator $\eta_\ell(\tilde{u}_\ell, \mathcal{T}_\ell)$ can be computed by
\begin{equation*}
\eta_\ell(\tilde{u}_\ell, \mathcal{T}_\ell) := \left( \sum_{\tau \in \mathcal{T}_\ell} \eta_\ell(\tilde{u}_\ell, \tau) \right)^{\frac{1}{2}},
\end{equation*}
and $\eta_\ell(\tilde{u}_\ell, \mathcal{T}_\ell)$ will be used as the stopping criterion for the overall AFEM algorithm.  Once it is smaller than a given tolerance, the AFEM algorithm will terminate.  In general, by the triangular inequality, we have
\begin{equation*}
\| \nabla u - \nabla \tilde{u}_\ell \| \leq \| \nabla \tilde{u}_\ell - \mathcal{G} \tilde{u}_\ell \| + \| \nabla u -  \mathcal{G} \tilde{u}_\ell \| = \eta_\ell(\tilde{u}_\ell, \mathcal{T}_\ell) + \| \nabla u -  \mathcal{G} \tilde{u}_\ell \|.
\end{equation*}
If the last term $\| \nabla u -  \mathcal{G} \tilde{u}_\ell \|$ is a high order term (which can be shown for second elliptic PDEs \cite{Xu.J;Zhang.Z2004}) compared with $\eta_\ell(\tilde{u}_\ell, \mathcal{T}_\ell)$, then we can expect that $\eta_\ell(\tilde{u}_\ell, \mathcal{T}_\ell)$ provides a good estimation of the true error $\| \nabla u - \nabla \tilde{u}_\ell \|$ and, therefore, guarantees the efficiency of the overall AFEM algorithm.  For FDES, numerical experiments presented below suggest that $\| \nabla u -  \mathcal{G} \tilde{u}_\ell \|$ is indeed a high order term.  A more rigorous analysis on this topic will be our future work.

\subsection{\texttt{MARK} Module}
The \texttt{MARK} module selects elements $\tau \in \mathcal{T}_\ell$ whose local error $\eta_\ell(U_\ell, \tau)$ is relatively large and needs to be refined in the refinement.  This module is independent of the model problems and we can directly use the strategies developed for second order elliptic PDEs for FDES here.  In this work, we use the so-called D\"{o}flers marking strategy \cite{Nochetto.R;Siebert.K;Veeser.A2009} with the detailed algorithm listed below (Algorithm 7). 

\begin{algorithm}
\caption{\texttt{MARK} Module: D\'{o}fler's marking strategy} \label{alg:mark}
$\mathcal{M}_\ell = \texttt{MARK}(\mathcal{T}_\ell, \eta_\ell(\tilde{u}_\ell, \tau), \theta)$
\begin{algorithmic}[1]
\STATE Choose a subset $\mathcal{M}_\ell \subset \mathcal{T}_\ell$ such that
\begin{equation} \label{def:dofler-mark}
\eta_\ell(\tilde{u}_\ell, \mathcal{M}_\ell) \leq \theta \eta_\ell(\tilde{u}_\ell, \mathcal{T}_\ell),
\end{equation}
where $\eta_\ell(\tilde{u}_\ell, \mathcal{M}_\ell) := \left( \sum_{\tau \in \mathcal{M}_\ell} \eta_\ell(\tilde{u}_\ell, \tau) \right)^{\frac{1}{2}}.
$
\end{algorithmic}
\end{algorithm}

Here we require that the parameter $\theta \in (0, 1]$.  Obviously, the choice of $\mathcal{M}_\ell$ is not unique.  In practice, in order to reduce the computational cost, we prefer the size of the subset $\mathcal{M}_\ell$ to be as small as possible.  Therefore, we typically use the greedy approach in the implementation of \texttt{ESTIMATE} module.  We first order the elements $\tau$ according to the error indicators $\eta_\ell(U_\ell, \tau)$ from large to small and then pick the element $\tau$ in a greedy way so that condition \eqref{def:dofler-mark} will be satisfied with minimal number of the elements.

Based on the above discussion, the ordering could be done in $\mathcal{O}(N_\ell \log N_\ell)$ operations and picking the elements can be done in $\mathcal{O}(N_\ell)$ operations, therefore, the overall computational complexity of the \texttt{ESTIMATE} module is $\mathcal{O}(N_\ell \log N_\ell)$.

\subsection{\texttt{REFINE} Module}
The \texttt{REFINE} module is also problem independent.  It takes the marked elements $\mathcal{M}_\ell$ and current grid $\mathcal{T}_k$ as inputs and outputs a refined grid $\mathcal{T}_{\ell+1}$, which will be used as a new grid.  In this work, because we are only considering the 1D case, the refinement procedure is just bisection.  Namely, if an element $[x_{i-1}^\ell, x_{i}^{\ell}] \in \mathcal{T}_\ell$ is marked, it will be divided into two subintervals $[x_{i-1}^{\ell}, \bar{x}_i^\ell]$ and $[\bar{x}_i^\ell, x_{i}^{\ell}]$ by the midpoint $\bar{x}_i^\ell = (x^\ell_{i-1} + x^\ell_i)/2$.  The detailed algorithm is listed below (Algorithm 8). 

\begin{algorithm}
\caption{\texttt{REFINE} Module: bisection refinement} \label{alg:refine}
$\mathcal{T}_{\ell+1} = \texttt{REFINE}(\mathcal{T}_\ell, \mathcal{M}_\ell)$
\begin{algorithmic}[1]
\FOR {$\tau \in \mathcal{M}_\ell$}
\STATE refine $\tau$ using bisection and generate two new elements.
\ENDFOR
\STATE Combine all new elements and subset $\mathcal{T}_\ell \backslash \mathcal{M}_l$ to generate the new grid $\mathcal{T}_{\ell+1}$.
\end{algorithmic}
\end{algorithm}

Obviously, the computational cost of the \texttt{REFINE} module is at most $\mathcal{O}(N_\ell)$.  

\subsection{AFEM Algorithm} \label{sec:AFEM}
After discussing each module, now we can summarize our AFEM algorithm for solving FDES.  We assume that an initial grid $\mathcal{T}_0$, a parameter $\theta \in (0,1]$, and a targeted tolerance $\varepsilon$ are given.  The AFEM algorithm is listed in Algorithm \ref{alg:AFEM}.

\begin{algorithm}
\caption{Adaptive Finite Element Method for Solving FDES} \label{alg:AFEM}
$\tilde{u}_J = \texttt{AFEM}(\mathcal{T}_0, \theta, \varepsilon)$ 
\begin{algorithmic}[1]
\STATE Set $\ell = 0$
\LOOP
\STATE $\tilde{u}_\ell = \texttt{SOLVE}(\mathcal{T}_\ell)$
\STATE $\{ \eta_\ell(\tilde{u}_\ell, \tau) \}_{\tau \in \mathcal{T}_\ell} = \texttt{ESTIMATE}(\tilde{u}_\ell, \mathcal{T}_\ell)$
\IF { $\eta_\ell(\tilde{u}_\ell, \mathcal{T}_\ell)  \leq \varepsilon$ }
\STATE $J = \ell$ and $\tilde{u}_J : = \tilde{u}_\ell$.
\RETURN
\ENDIF
\STATE $\mathcal{M}_\ell = \texttt{MARK}(\mathcal{T}_\ell, \eta_\ell(\tilde{u}_\ell, \tau), \theta)$
\STATE $\mathcal{T}_{\ell+1} = \texttt{REFINE}(\mathcal{T}_\ell, \mathcal{M}_\ell)$
\STATE $\ell = \ell+1$
\ENDLOOP
\end{algorithmic}
\end{algorithm}

Obviously, the overall computational cost of each iteration of the AFEM method is $\mathcal{O}(k N_\ell \log N_\ell)$ or $\mathcal{O}(N_\ell \log N_\ell)$ when $k \ll N_\ell$.

\section{Numerical Examples} \label{sec:numerics}
In this section, we present some numerical experiments to demonstrate the efficiency and robustness of the proposed $\mathcal{H}$-Matrix representation, GMG method, and AFEM algorithm for solving the FDES.  All the codes are written in MATLAB and the tests are performed on a Macbook Pro Laptop with Inter Core i7 (3 GHz) CPU and 16G RAM.

 \begin{example} \label{exp1}
 Solving problem \eqref{Pro1} with exact solution $u(x)=10x^2(1-x^2)$ on [b,c]=[0,1].  The right hand side can be computed as
\begin{align*}
f(x)=&-\frac{10}{2\cos(\af\pi/2)}\left\{\frac{2}{\Gamma(3-\af)}\left[x^{2-\af}+(1-x)^{2-\af}\right]-\frac{12}{\Gamma(4-\af)}\left[x^{3-\af}+(1-x)^{3-\af}\right] \right. \\
&\left.+
\frac{24}{\Gamma(5-\af)}\left[x^{4-\af}+(1-x)^{4-\af}\right]  \right\}.
\end{align*}
\end{example}
It is easy to see that for Example \ref{exp1}, the solution $u(x)$ is smooth.  Therefore, the adaptive method is unnecessary for this example and we mainly use uniform grid and uniform refinement.  The purpose of this example is to show the accuracy of the $\mathcal{H}$-Matrix representation, the efficiency of the GMG method, and the overall optimal computational complexity of the proposed approach.   

\begin{figure}[ht!]
\caption{Example \ref{exp1} ($\alpha = 1.2$). Convergence comparison between full matrix approach and $\mathcal{H}$-Matrix representation.} \label{fig:example1-alpha12}
\centering
\includegraphics[width=0.48\textwidth]{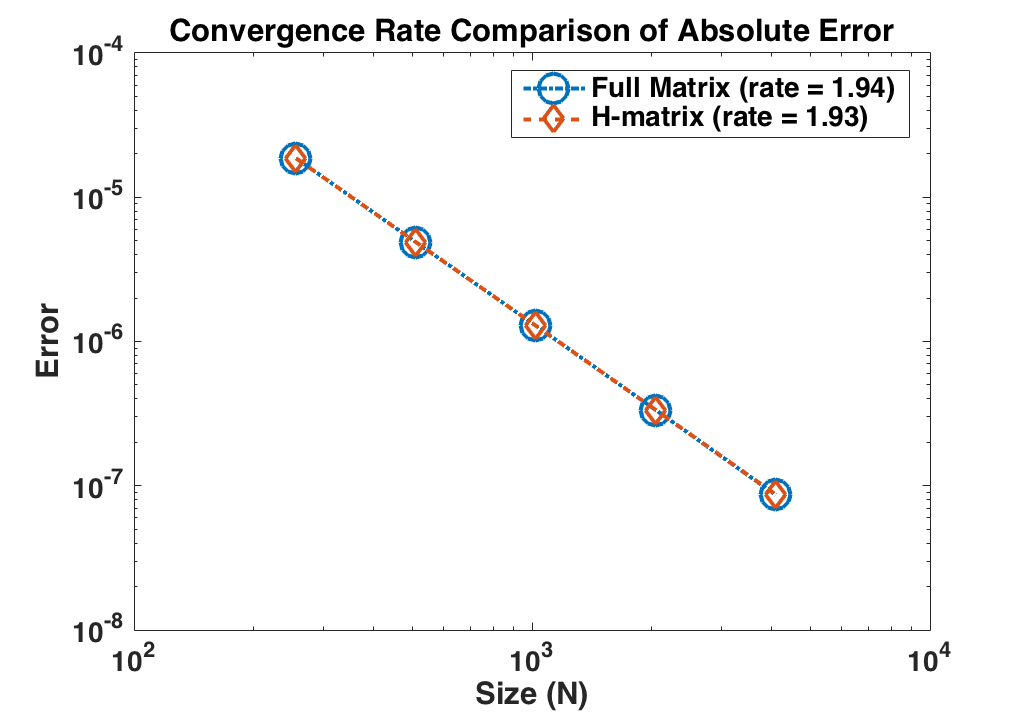} \
\includegraphics[width=0.48\textwidth]{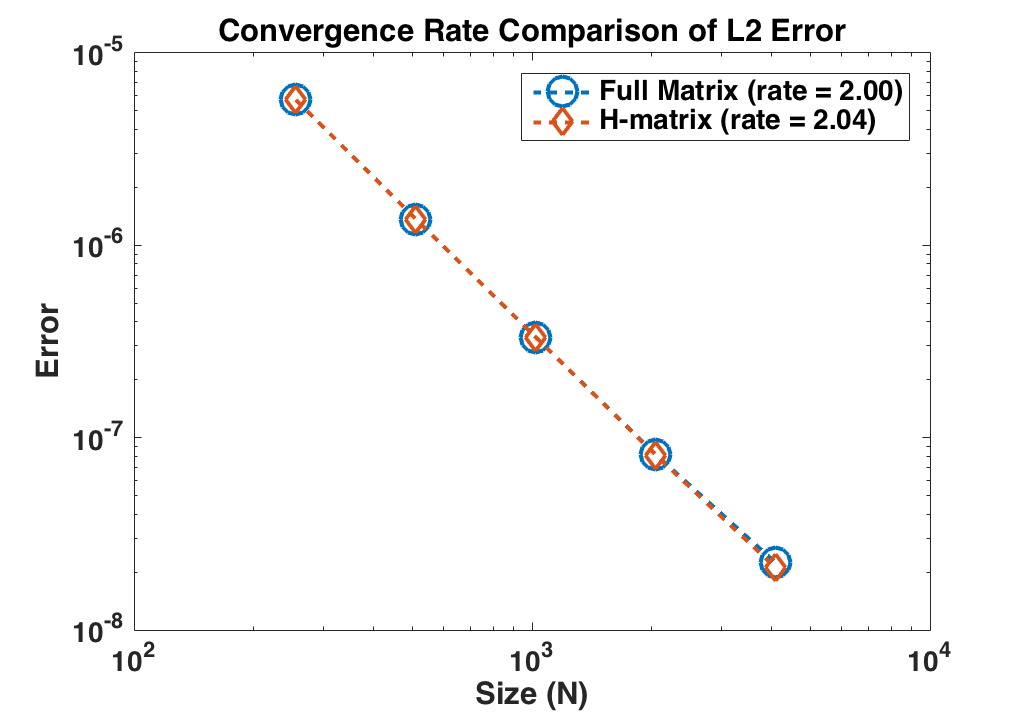}
\end{figure}

\begin{figure}[ht!]
\caption{Example \ref{exp1} ($\alpha = 1.5$). Convergence comparison between full matrix approach and $\mathcal{H}$-Matrix representation.} \label{fig:example1-alpha15}
\centering
\includegraphics[width=0.48\textwidth]{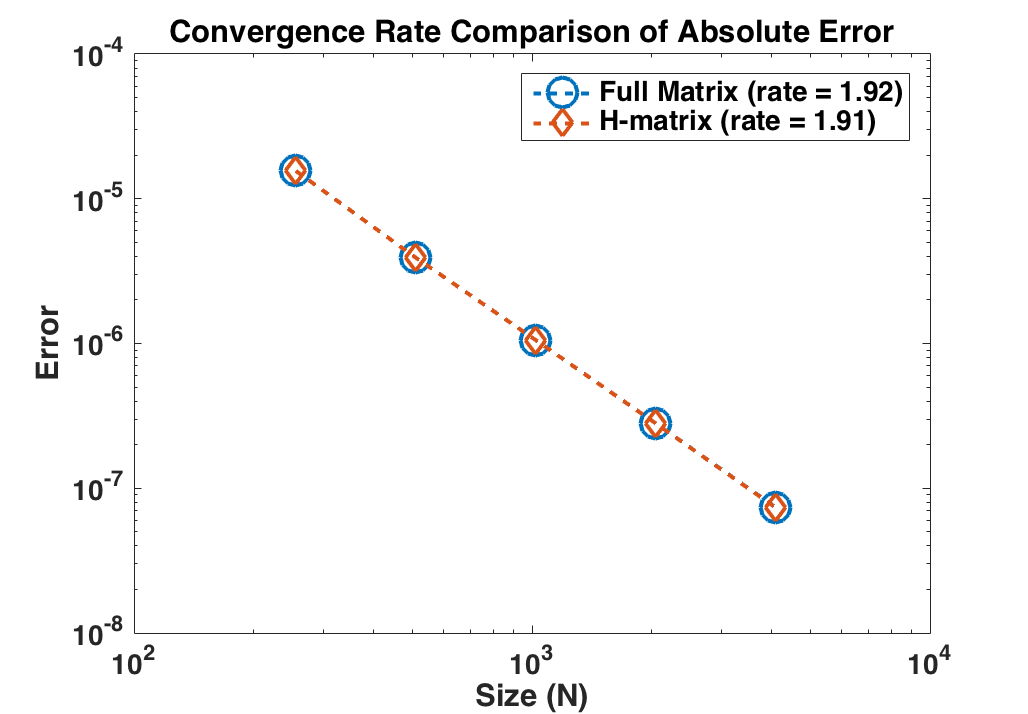} \
\includegraphics[width=0.48\textwidth]{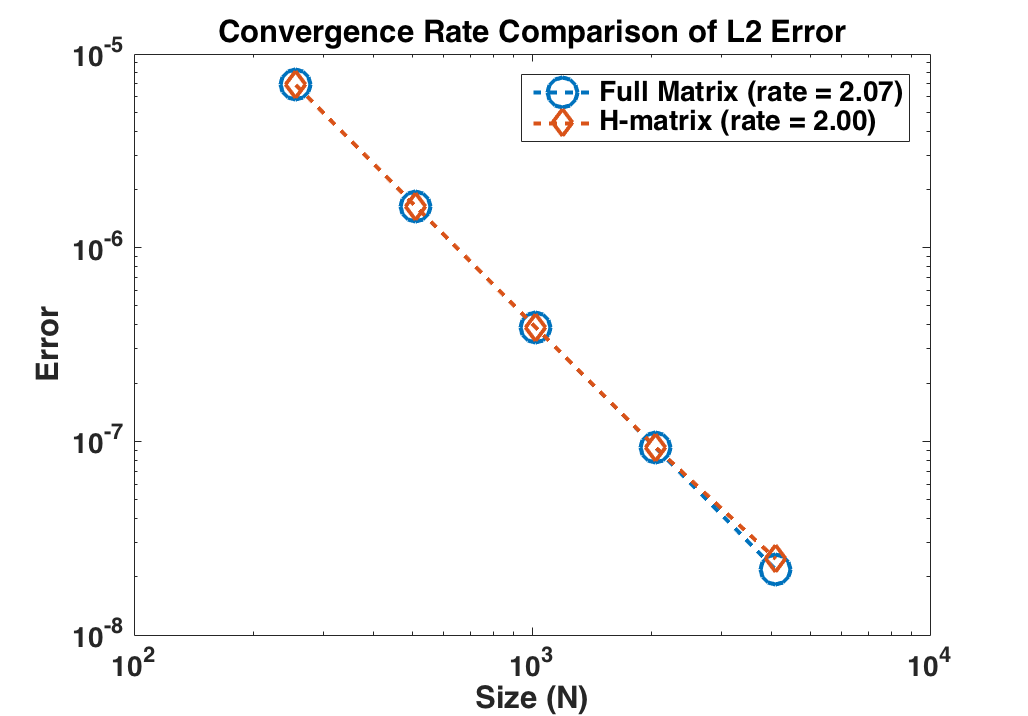}
\end{figure}


Figures \ref{fig:example1-alpha12} and \ref{fig:example1-alpha15} present the convergence behavior of the finite element approximations based on full matrix approach and $\mathcal{H}$-Matrix representation for different values of $\alpha$.  For the full matrix approach, we use a direct solver (LU decomposition) to solve the linear system of equations (command ``$\backslash$" in MATLAB) and, for the $\mathcal{H}$-Matrix, we solve the linear system of equations iteratively by the GMG method presented in Section \ref{sec:MG-Hmat}.  The stopping criterion is the relative residual to be less than $10^{-10}$.  The convergence rate of absolute error and $L^2$ error are presented.  In all our experiments, we use the fact that $\log E = -r \log N + \log C$ which is derived from $E = C N^{-r}$, where $E$ denotes error, and then compute the convergence rate $r$ by the linear polynomial fitting between $\log E$ and $\log N$.  In all cases, we can see that using the $\mathcal{H}$-Matrix representation, the convergence orders are still around $2$ which is optimal as expected.  Moreover, in all cases, the errors obtained by the $\mathcal{H}$-Matrix representation are also comparable with the errors obtained by using the full matrix.  The results show that using the $\mathcal{H}$-Matrix representation can still achieve the optimal convergence order and the accuracy of finite element approximations is still reliable. 

As mentioned before, the advantages of using $\mathcal{H}$-Matrix are not only the accuracy but also the fact that it significantly reduces the computational cost compared with the full matrix, especially when the GMG method is applied.  In Table \ref{tab:MG-iter}, the number of iterations for GMG method is shown for different mesh size $h$ and fractional index $\alpha$.  We can see that the number of iterations is quite stable for wide ranges of parameters $h$ and $\alpha$, which demonstrates the optimal convergence and robustness of the proposed GMG method in $\mathcal{H}$-Matrix format.  Moreover, in Figure \ref{fig:MG-compare}, we compare the CPU time of the LU decomposition for the full matrix and the GMG method for the $\mathcal{H}$-Matrix representation.  We can see that the computational cost of the GMG method behaves like $\mathcal{O}(N^{0.92})$, which is significantly better than the computational cost of the LU decomposition for full matrix ($\mathcal{O}(N^{2.68})$).  Theoretically, the GMG method based on the $\mathcal{H}$-Matrix representation costs $\mathcal{O}(N \log N)$.  Therefore, we can expect even bigger speedup when the problem size $N$ is increased.  We want to comment that, in Figure \ref{fig:MG-compare}, for relative small $N$, LU decomposition for the full matrix seems to be faster than the GMG method for the $\mathcal{H}$-Matrix.  This is because of the different implementations of LU decomposition and the GMG methods.  For LU decomposition, Matlab build-in command ``$\backslash$" is used which is based on the UMFPack package implemented in the Matlab. Our GMG method for the $\mathcal{H}$-Matrix is completely implemented in Matlab.  Our Matlab implementation actually outperforms the build-in command ``$\backslash$" for large size $N$, which is a strong evidence of the efficiency of the GMG method for the $\mathcal{H}$-Matrix.  
  
\begin{table}[htp]
\caption{Example \ref{exp1}: number of iterations of GMG method (stopping criterion: relative residual less than or equal to $10^{-10}$)} \label{tab:MG-iter}
\begin{center}
\begin{tabular}{|c|c|c|c|c|c|}
\hline \hline
		       & $h=1/256$ & $h=1/512$ & $h=1/1024$ & $h=1/2048$ & $h=1/4095$ \\ \hline
$\alpha = 1.1$ & 9 & 9 & 9 & 9 & 9 \\
$\alpha = 1.3$ & 10 & 10 & 10 & 10 & 11 \\
$\alpha = 1.5$ & 11 & 11 & 11 & 12 & 12 \\
$\alpha = 1.7$ & 12 & 12 & 12 & 12 & 13 \\
$\alpha = 1.9$ & 13 & 13 & 13 & 13 & 14 \\ \hline \hline
\end{tabular}
\end{center}
\end{table}%

\begin{figure}[ht!]
\caption{Example \ref{exp1} ($\alpha = 1.5$) CPU time comparison between full matrix (LU decomposition) and the $\mathcal{H}$-Matrix (multigrid)} \label{fig:MG-compare-alpha15} \label{fig:MG-compare}
\centering
\includegraphics[width=0.65\textwidth]{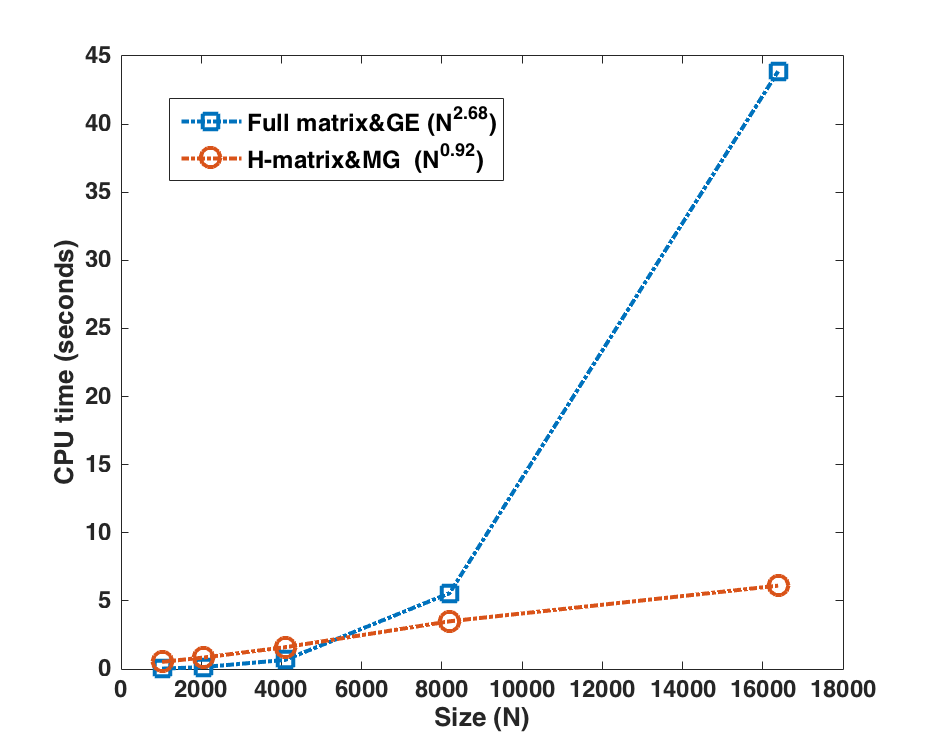}
\end{figure}

\begin{example} \label{exp:2}
Solving model problem \eqref{Pro1} with  $f(x) = -(1 + \sin(x))$.
\end{example}

The second example we consider here does not have exact solution.  However, due to the property of the FDES, we expect the solution to have singularities near the boundaries, which leads to degenerated convergence rate in the errors of finite element approximations on uniform grids.  This is confirmed by the numerical results as shown in Figure \ref{fig:exp2-convergence-rate-alpha12}.  The convergence rates of the $L^2$ errors for both full matrix and $\mathcal{H}$-Matrix approaches are about $1.2$, which reflects the singularity of the solution and the necessity for the AFEM method.  These comparisons show that the AFEM algorithm can achieve better accuracy with less computational cost, which demonstrates the effectiveness of the AFEM algorithm for FDES. 

\begin{figure}[ht!] \centering
\caption{Example \ref{exp:2} ($\alpha = 1.2$): Convergence rate comparison between full matrix and $\mathcal{H}$-Matrix representation on uniform grids.} \label{fig:exp2-convergence-rate-alpha12}
\includegraphics[width=0.6\linewidth]{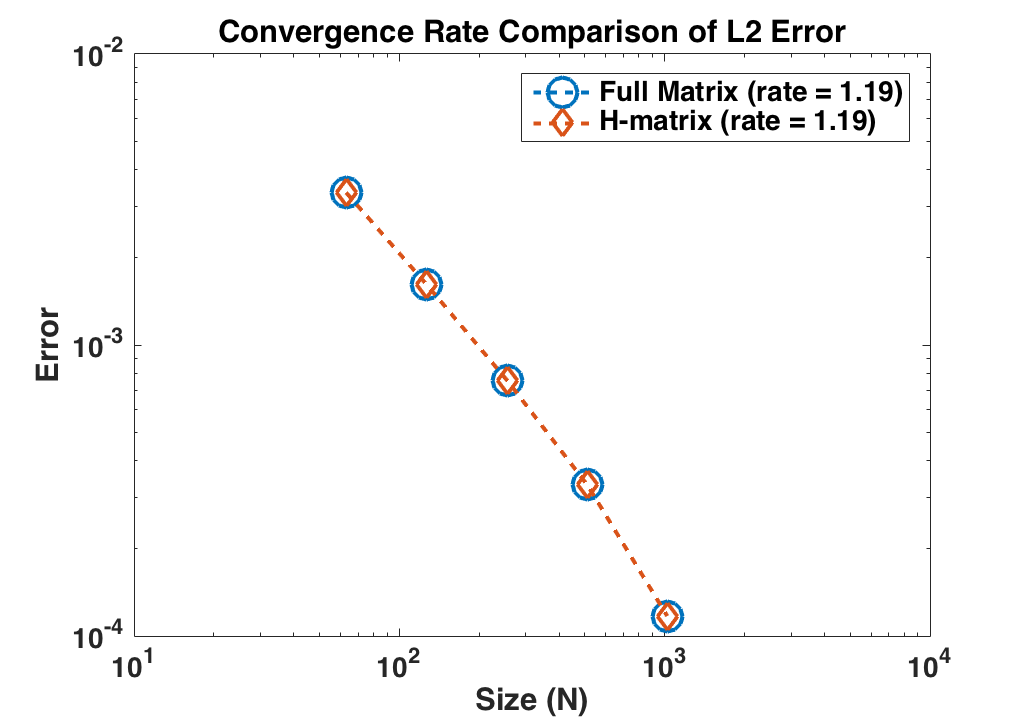}
\end{figure}

Next we apply the AFEM algorithm (Algorithm \ref{alg:AFEM}) to solve Example \ref{exp:2}.  The results are shown in Figures \ref{fig:exp2-AFEM-L2} and \ref{fig:exp2-AFEM-Linfty}.  We can see that, using the AFEM method, the optimal convergence rates of both $L^2$ error and $L^{\infty}$ have been recovered for both $\alpha = 1.3$ and $\alpha = 1.5$.  This demonstrates the effectiveness and robustness of the our AFEM methods.  In Figure \ref{fig:exp2-AFEM-mesh}, we plot the numerical solutions on adaptive meshes for both $\alpha = 1.3$ and $\alpha = 1.5$.  The adaptive refinement near the boundary points demonstrates that our error estimates captures the singularities well and overall robustness of our AFEM algorithm.

\begin{figure}[ht!] \centering
\caption{Example \ref{exp:2}: Convergence rate of $L^2$ errors comparison between FEM on uniform grid and AFEM (Left: $\alpha = 1.3$; Right: $\alpha = 1.5$).} \label{fig:exp2-AFEM-L2}
\includegraphics[width=0.48\linewidth]{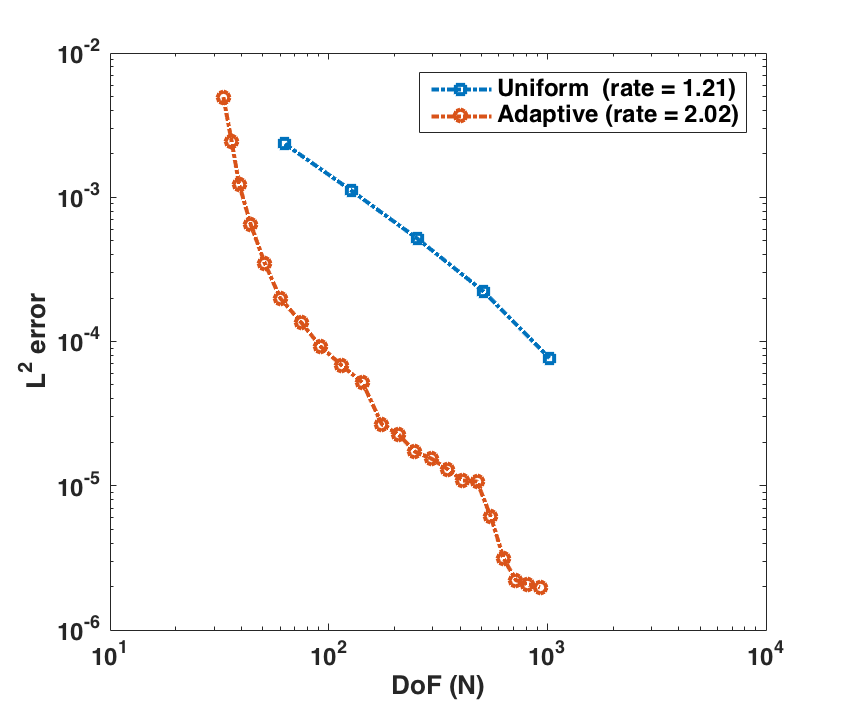}
\includegraphics[width=0.495\linewidth]{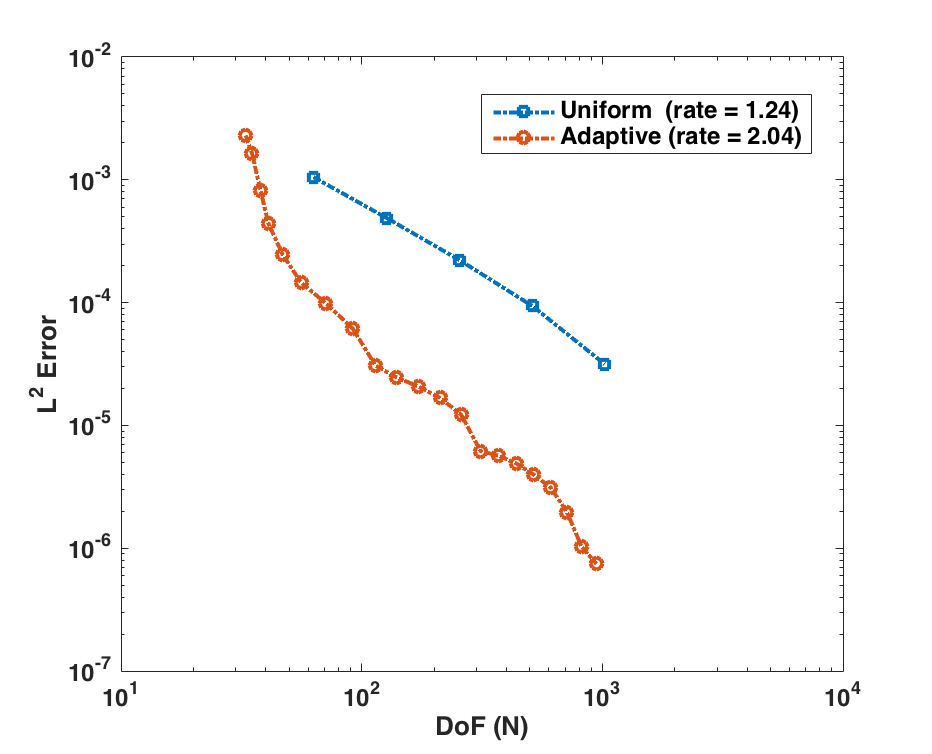}
\end{figure}

\begin{figure}[ht!] \centering
\caption{Example \ref{exp:2}: Convergence rate of $L^{\infty}$ errors comparison between FEM on uniform grid and AFEM (Left: $\alpha = 1.3$; Right: $\alpha = 1.5$).} \label{fig:exp2-AFEM-Linfty}
\includegraphics[width=0.48\linewidth]{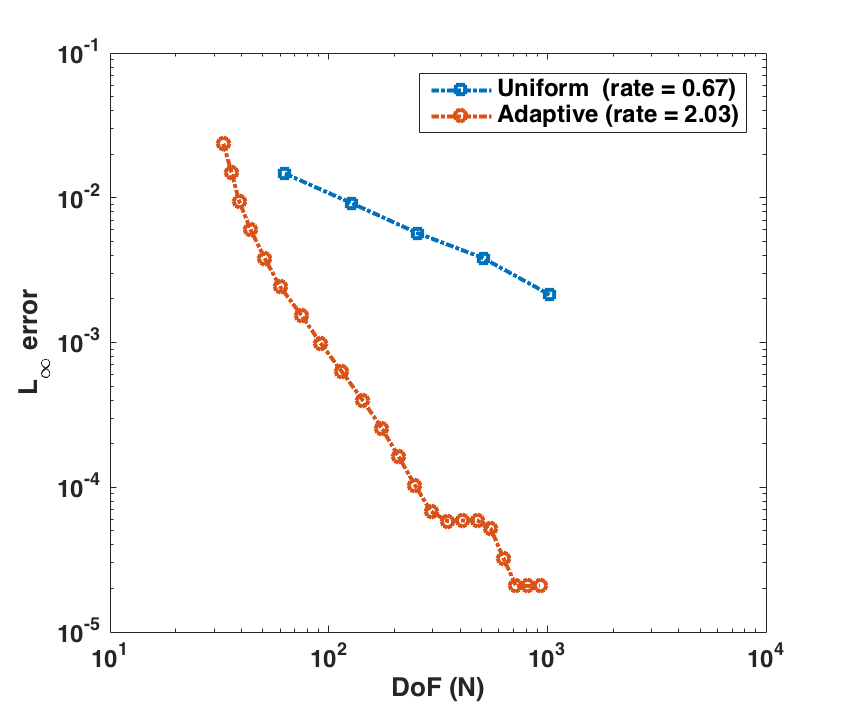}
\includegraphics[width=0.495\linewidth]{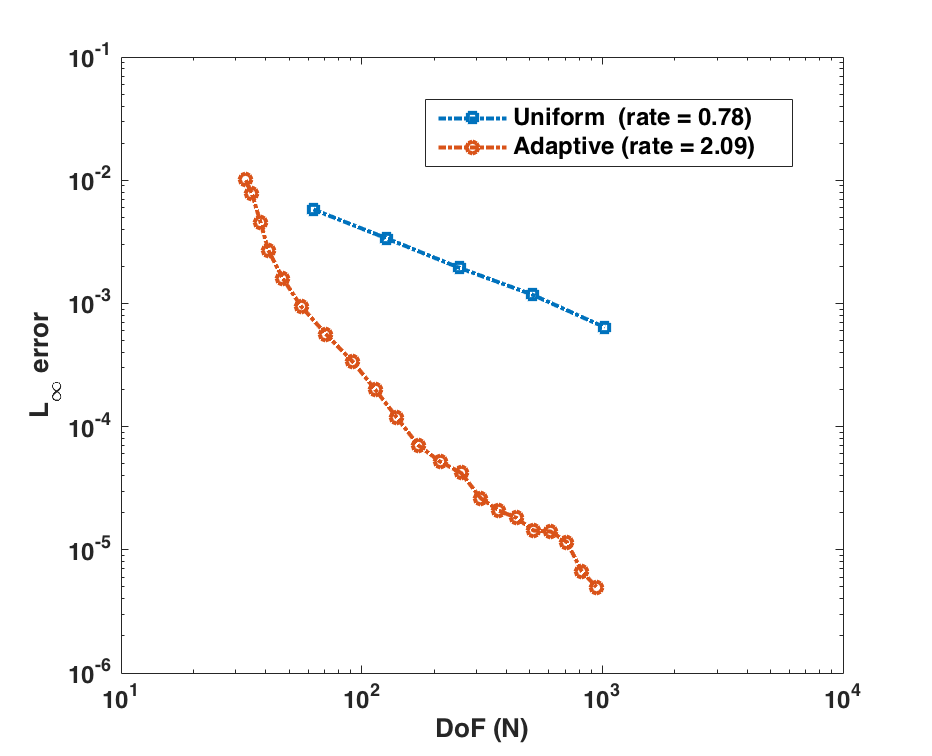}
\end{figure}

\begin{figure}[ht!] \centering
\caption{Example \ref{exp:2}: Numerical solutions on adaptive meshes (Left: $\alpha = 1.3$; Right: $\alpha = 1.5$).} \label{fig:exp2-AFEM-mesh}
\includegraphics[width=0.48\linewidth]{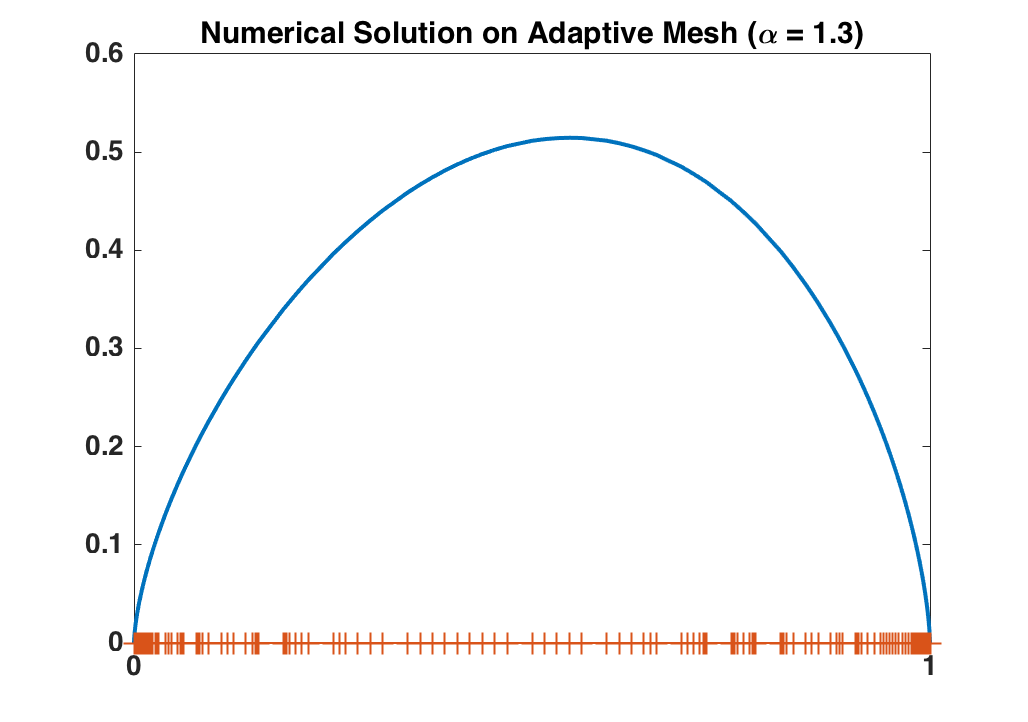} \
\includegraphics[width=0.48\linewidth]{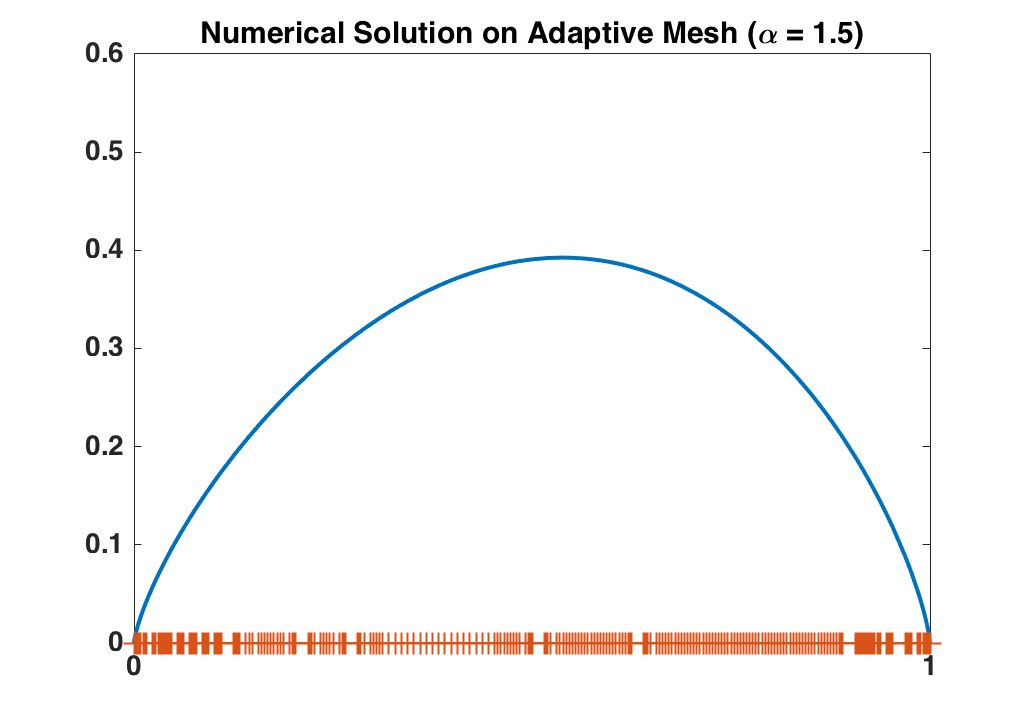}
\end{figure}

As mentioned in Section \ref{sec:AFEM-FPDEs}, one distinct feature of our proposed AFEM method is that in the \texttt{SOLVE} module, the
$\mathcal{H}$-Matrix representation and the multigrid method are used, hence providing nearly optimal computational complexity $\mathcal{O}(N \log N)$.  In Figure \ref{fig:GMG-AFEM-compare}, we show the CPU time of the GMG method for the $\mathcal{H}$-Matrix used in the \texttt{SOLVE} module for different fractional orders.  We can see that, for all cases, the computational complexity is optimal, which confirms our expectation.

\begin{figure}
\centering
\caption{Example \ref{exp:2}: CPU time of GMG method for $\mathcal{H}$-Matrix on nonuniform grid (different fractional index $\alpha$)} \label{fig:GMG-AFEM-compare}
\includegraphics[width=0.7\linewidth]{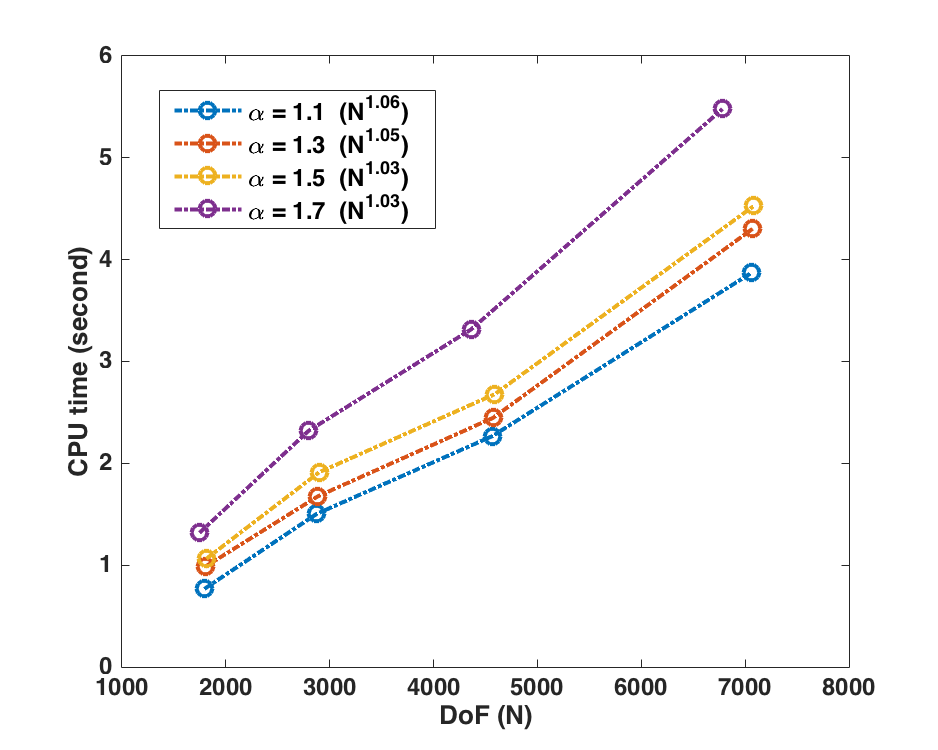}
\end{figure}

Next we compare the computational costs of FEM on uniform grids and AFEM.  The results are shown in Table \ref{tab:uniform-adaptive-time-compare}.  Here ``DoFs" means the degrees of freedom.  For AFEM, we start the adaptive refinement from a coarse grid of size $32$.  For AFEM, ``Total DoFs" means the sum of the DoFs of all the adaptive grids starting from the coarse grid to current adaptive grid and ``Total Time" means the total CPU time of the whole AFEM algorithm while ``Time" means the CPU time of solving the FDES on the current adaptive grid.  For FEM on a uniform grid of size $16,383$, the $L^2$ error is $1.89 \times 10^{-6}$ and the CPU time is about $49.31$ seconds.  However, if we use AFEM, solving the problem on an adaptive grid of size $1,059$ leads to accuracy $7.07 \times 10^{-7}$ in the $L^2$ norm. This means that the AFEM achieves about $2.7$ times better results using a $15.5$ times smaller grid.  Even the total DoFs, which is $6,382$, is about $2.6$ times smaller than the size of the uniform grid.  The speed up is about $18.3$ if we only consider the final adaptive grid and is about $3.2$ if we consider the whole AFEM procedure.  

\begin{table}[htp]
\caption{Example \ref{exp:2} ($\alpha = 1.5$): Computational cost comparison between FEM on uniform grids and AFEM. (The time unit is second)} \label{tab:uniform-adaptive-time-compare}
\centering
\begin{tabular}{|c|c|c||c|c|c|c|c|}
\hline \hline
\multicolumn{3}{|c||}{Uniform} & \multicolumn{5}{|c|}{Adaptive} \\ \hline 
$L^2$ error & DoFs & Time & $L^2$ error & DoFs & Time & Total DoFs & Total Time \\ \hline \hline
$5.84 \times 10^{-5}$ & $1,023$ & $3.55$ & $3.40 \times 10^{-5}$ & $116$ & $0.28$ & $490$ & $1.03$ \\ 
$1.30 \times 10^{-5}$ & $4,095$ & $13.57$ & $9.86 \times 10^{-6}$ & $272$ & $0.59$ & $1,312$ & $2.82$ \\ 
$1.89 \times 10^{-6}$ & $16,383$ & $49.31$ & $7.07 \times 10^{-7}$ & $1,059$ & $2.70$ & $6,382$ & $15.52$ \\ 
\hline \hline
\end{tabular}
\end{table}%

In Figure \ref{fig:AFEM-breakdown}, we report the breakdown of computational cost of the AFEM on the finest adaptive grid.  As we can see, the \texttt{SOLVE} module (Assembling and $\mathcal{H}$-Matrix \& MG Solve) dominates the whole AFEM algorithm.  The computational cost of the other three modules, \texttt{ESTIMATE} module (Estimate Error), \texttt{MARK} module (Mark), and \texttt{REFINE} module (Adaptive Refine), are roughly the same and could be ignored compared with the \texttt{SOLVE} module.  This is mainly because our experiments are in 1D in the current paper.  We can expect those three modules to become more and more time consuming in 2D and 3D cases.  However, the \texttt{SOLVE} module should still be the dominant module in terms of computational cost and that is why we introduce the $\mathcal{H}$-Matrix and the GMG method together to make sure that we achieve optimal computational complexity.  

\begin{figure}
\centering
\caption{Example \ref{exp:2} ($\alpha = 1.5$): Breakdown of CPU time of AFEM }
\label{fig:AFEM-breakdown}
\smallskip
\includegraphics[width=0.8\linewidth]{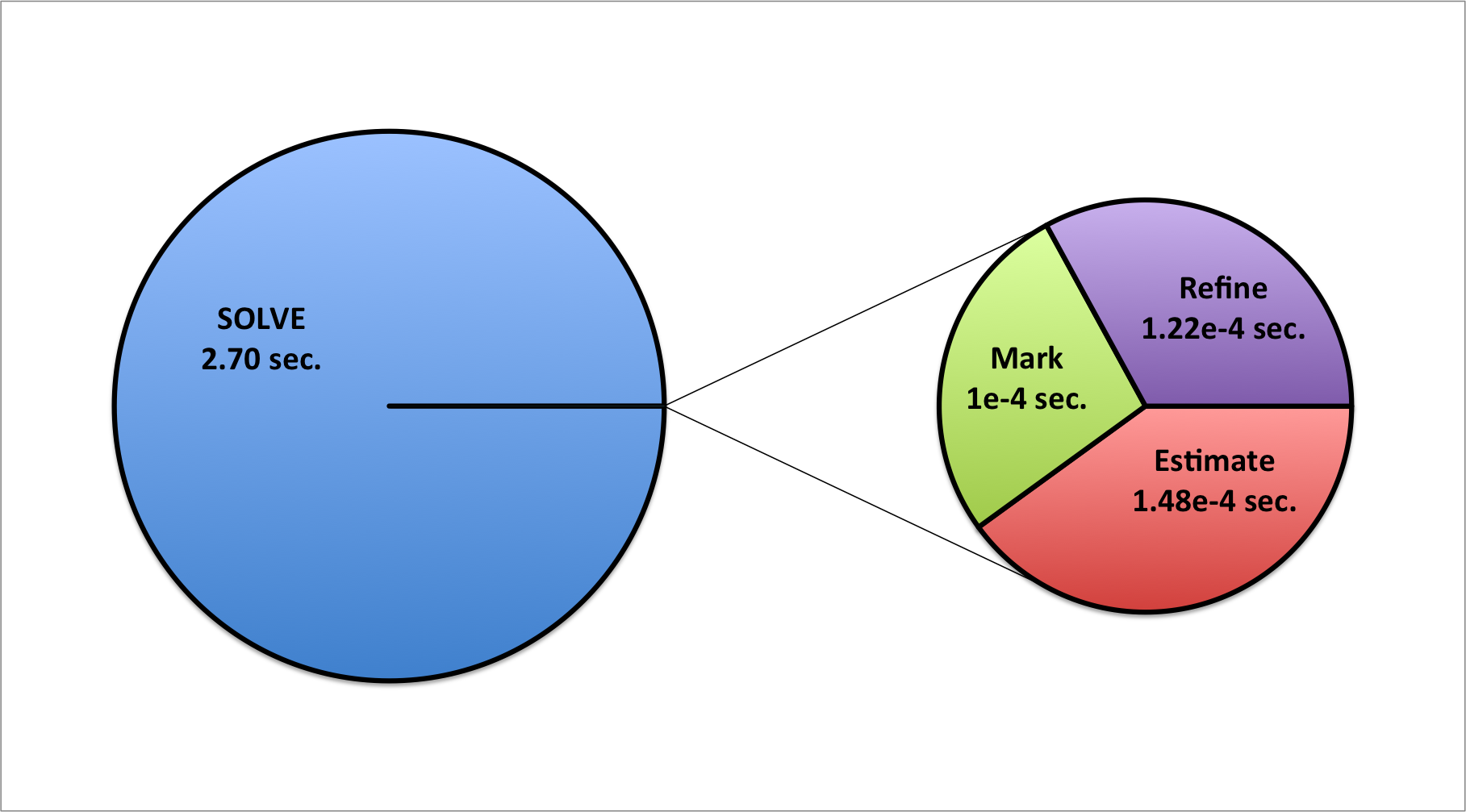}
\end{figure}

\begin{example} \label{exp:3}
We consider the following fractional boundary value problem with endpoint singularities
\begin{align}
&-(\kappa_1 \prescript{RL}{0}{\mathcal{D}}_{x}^{\af}+\kappa_2 \prescript{RL}{x}{\mathcal{D}}_{1}^{\af})u(x)=1+\sin(x),\quad x\in(0,1),\\
&u(0)=0,\quad u(1)=0.
\end{align}
\end{example}
If we choose $\kappa_1=\kappa_2=\frac1{2\cos(\af \pi/2)},$ Example \ref{exp:3} is the same as  Example \ref{exp:2}. The previous numerical  results are moderately good due to the symmetry of the Riesz fractional operator, which leads to the partial cancellation of the singularity of the solution.  In this example, we investigate further the singularities induced by the fractional operators by choosing $\kappa_1/\kappa_2$ sufficiently large (or small) which leads to stronger singularity at the left (or right) endpoint.  More precisely, we fix $\alpha = 1.5$ and choose two sets of $\kappa_1$ and $\kappa_2$, i.e. $\kappa_1 = \frac{1}{2\cos(\af \pi/2)}$ and $\kappa_2 = 0$, $\kappa_1 = \frac{1}{2\cos(\af \pi/2)}$ and $\kappa_2 = \frac{0.1}{2\cos(\af \pi/2)}$.

\begin{figure}[ht!] \centering
\caption{Example \ref{exp:3}: $L^2$ errors versus number of degrees-of-freedom and the comparison between FEM on uniform grid and AFEM (Left: $\kappa_1 = \frac{1}{2\cos(\af \pi/2)}$ and $\kappa_2 = 0$; Right: $\kappa_1 = \frac{1}{2\cos(\af \pi/2)}$ and $\kappa_2 = \frac{0.1}{2\cos(\af \pi/2)}$).} \label{fig:exp3-AFEM-L2}
\includegraphics[width=0.48\linewidth]{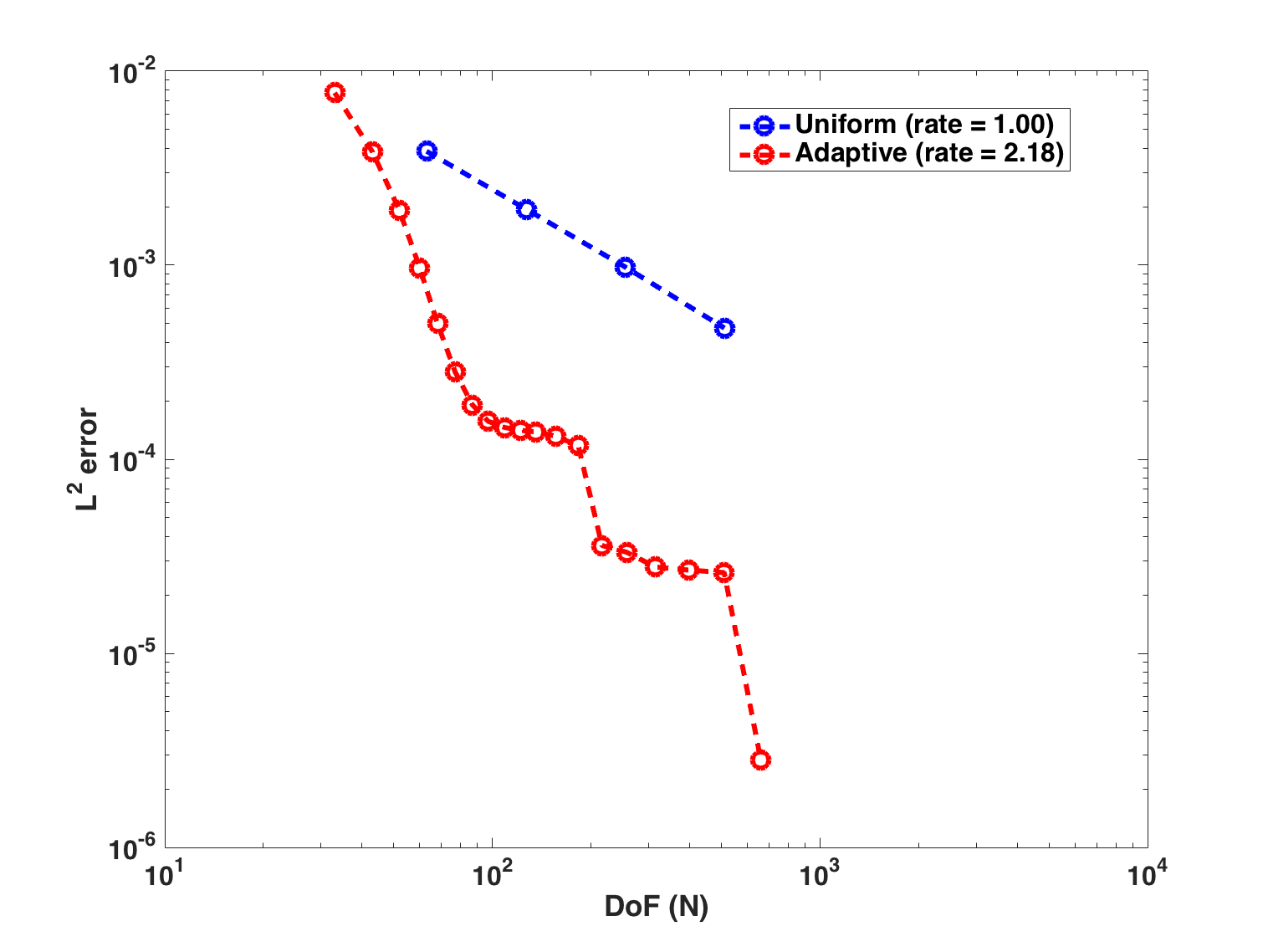}
\includegraphics[width=0.48\linewidth]{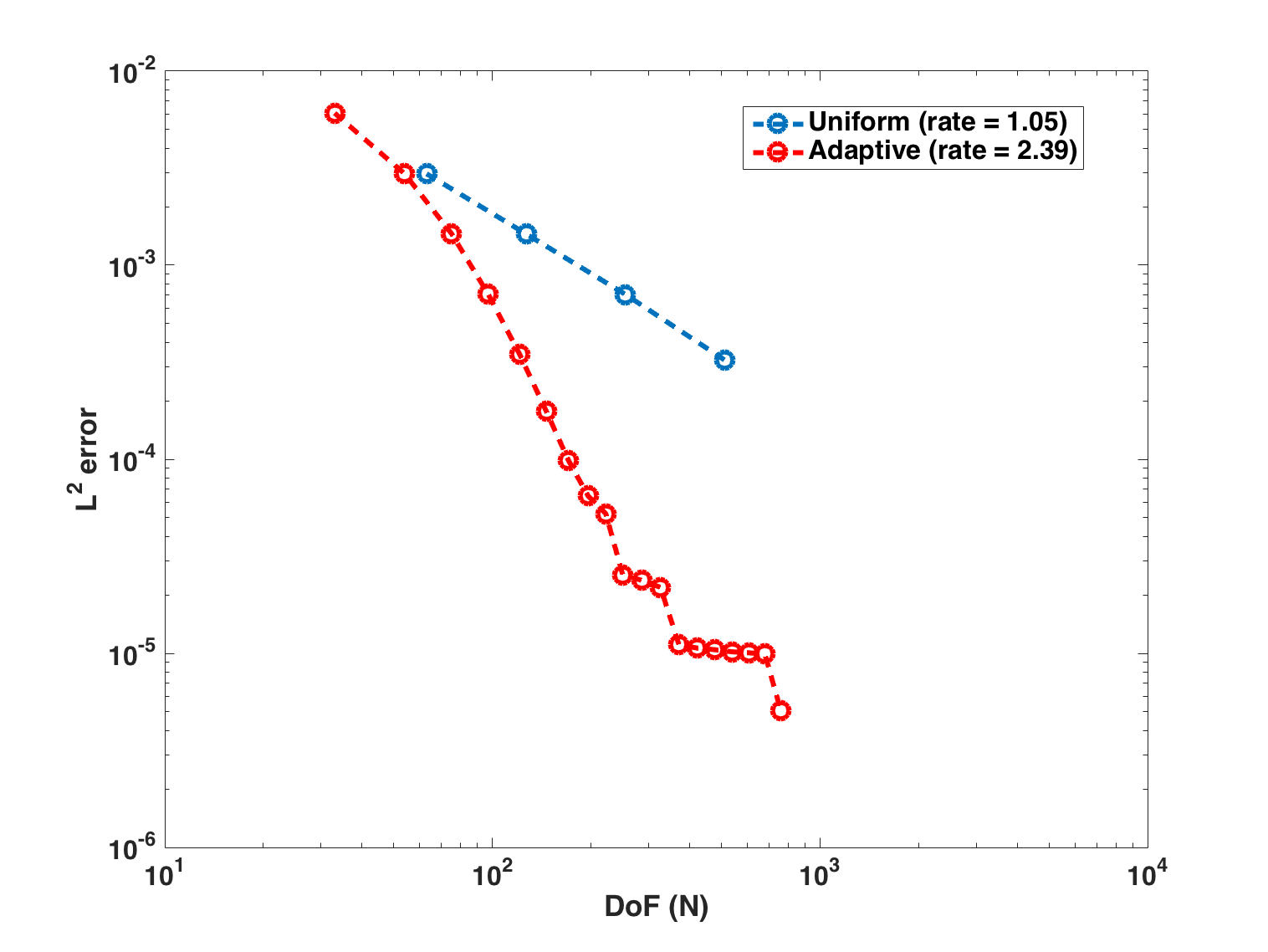}
\end{figure}

\begin{figure}[ht!] \centering
\caption{Example \ref{exp:3}: $L^{\infty}$ errors versus number of degrees-of-freedom and the comparison between FEM on uniform grid and AFEM (Left: $\kappa_1 = \frac{1}{2\cos(\af \pi/2)}$ and $\kappa_2 = 0$; Right: $\kappa_1 = \frac{1}{2\cos(\af \pi/2)}$ and $\kappa_2 = \frac{0.1}{2\cos(\af \pi/2)}$).} \label{fig:exp3-AFEM-Linfty}
\includegraphics[width=0.48\linewidth]{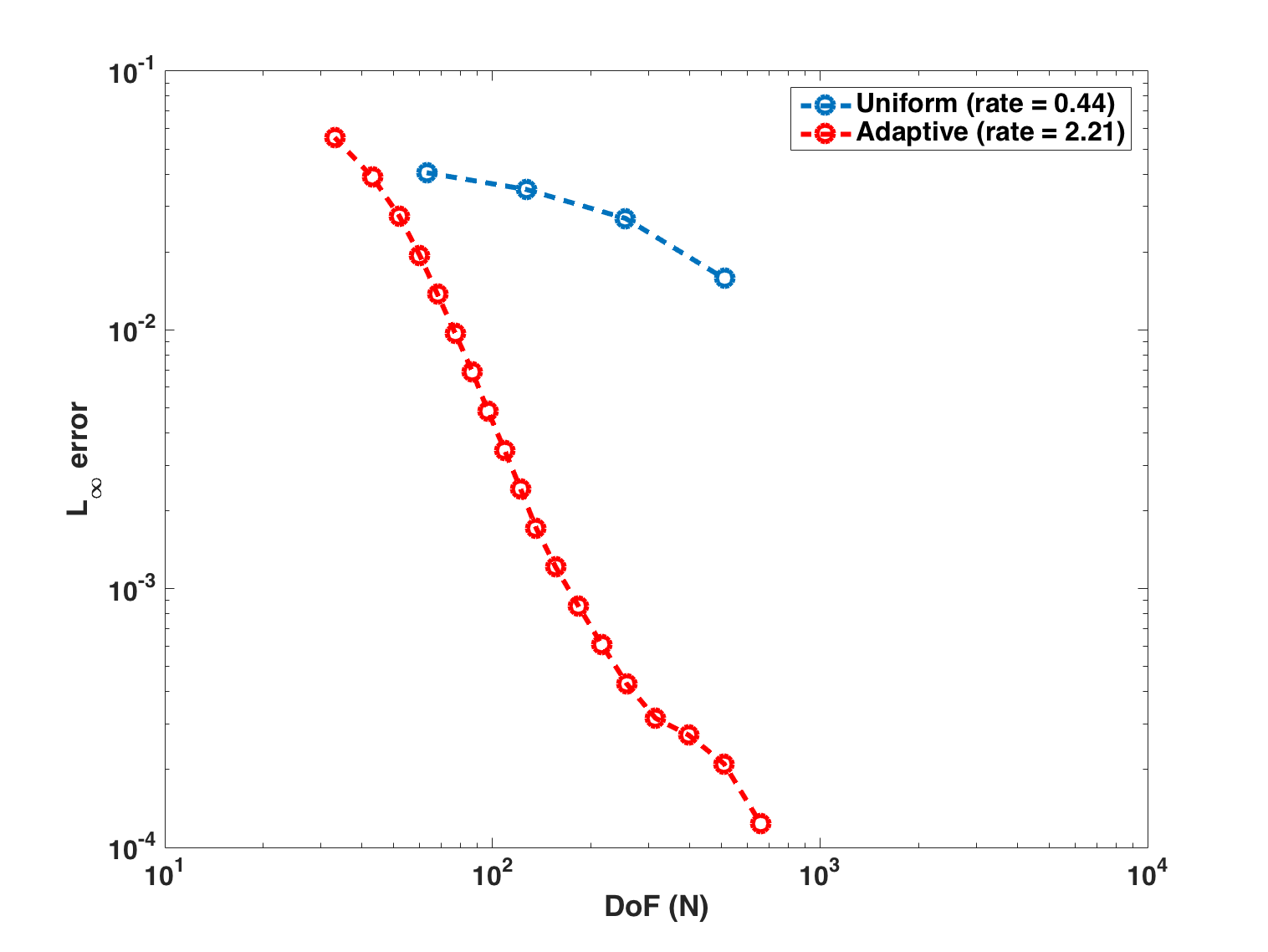}
\includegraphics[width=0.48\linewidth]{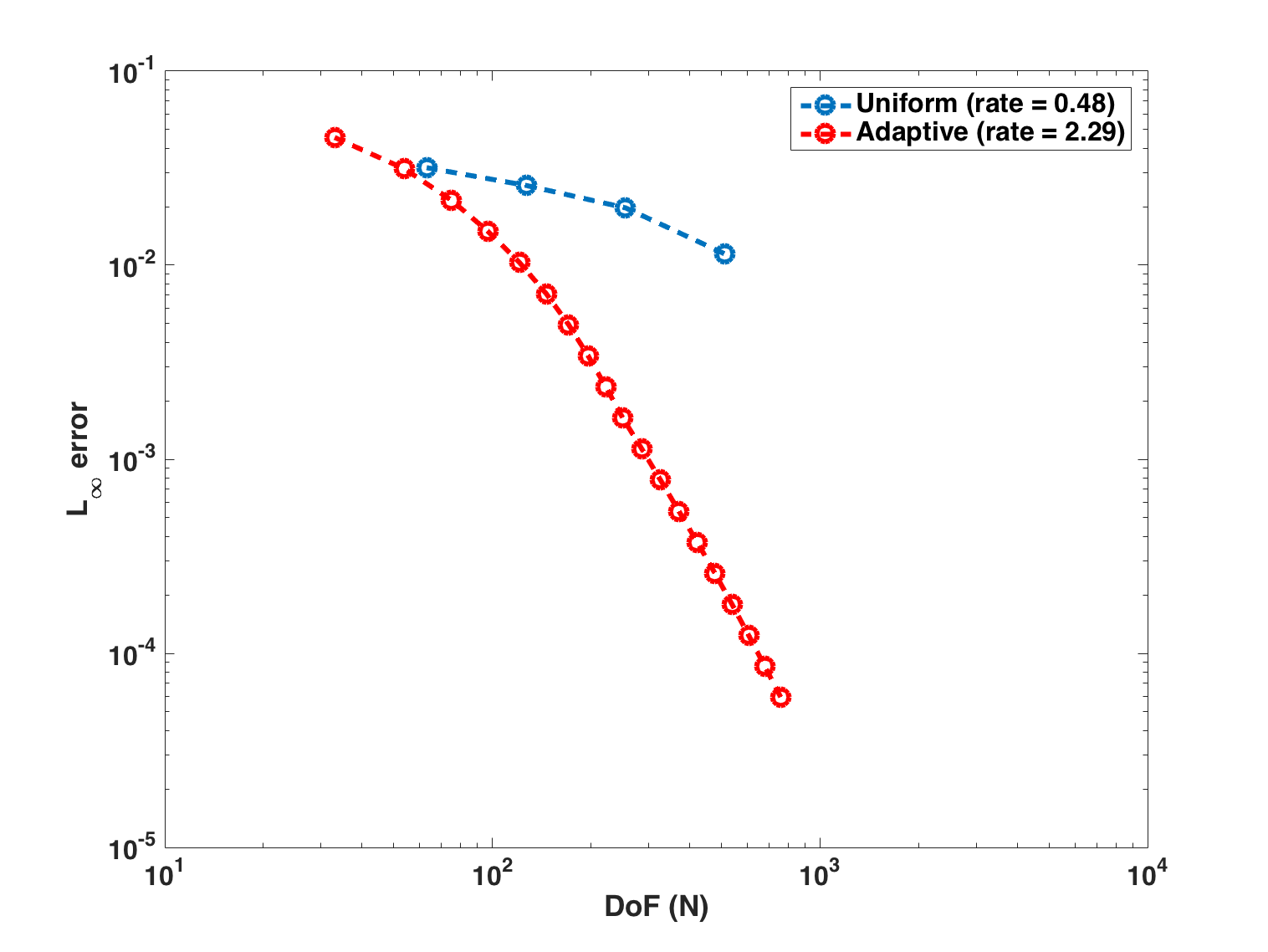}
\end{figure}

\begin{figure}[ht!] \centering
\caption{Example \ref{exp:2}: Numerical solutions on adaptive meshes (Left: $\kappa_1 = \frac{1}{2\cos(\af \pi/2)}$ and $\kappa_2 = 0$; Right: $\kappa_1 = \frac{1}{2\cos(\af \pi/2)}$ and $\kappa_2 = \frac{0.1}{2\cos(\af \pi/2)}$).} \label{fig:exp3-AFEM-mesh}
\includegraphics[width=0.48\linewidth]{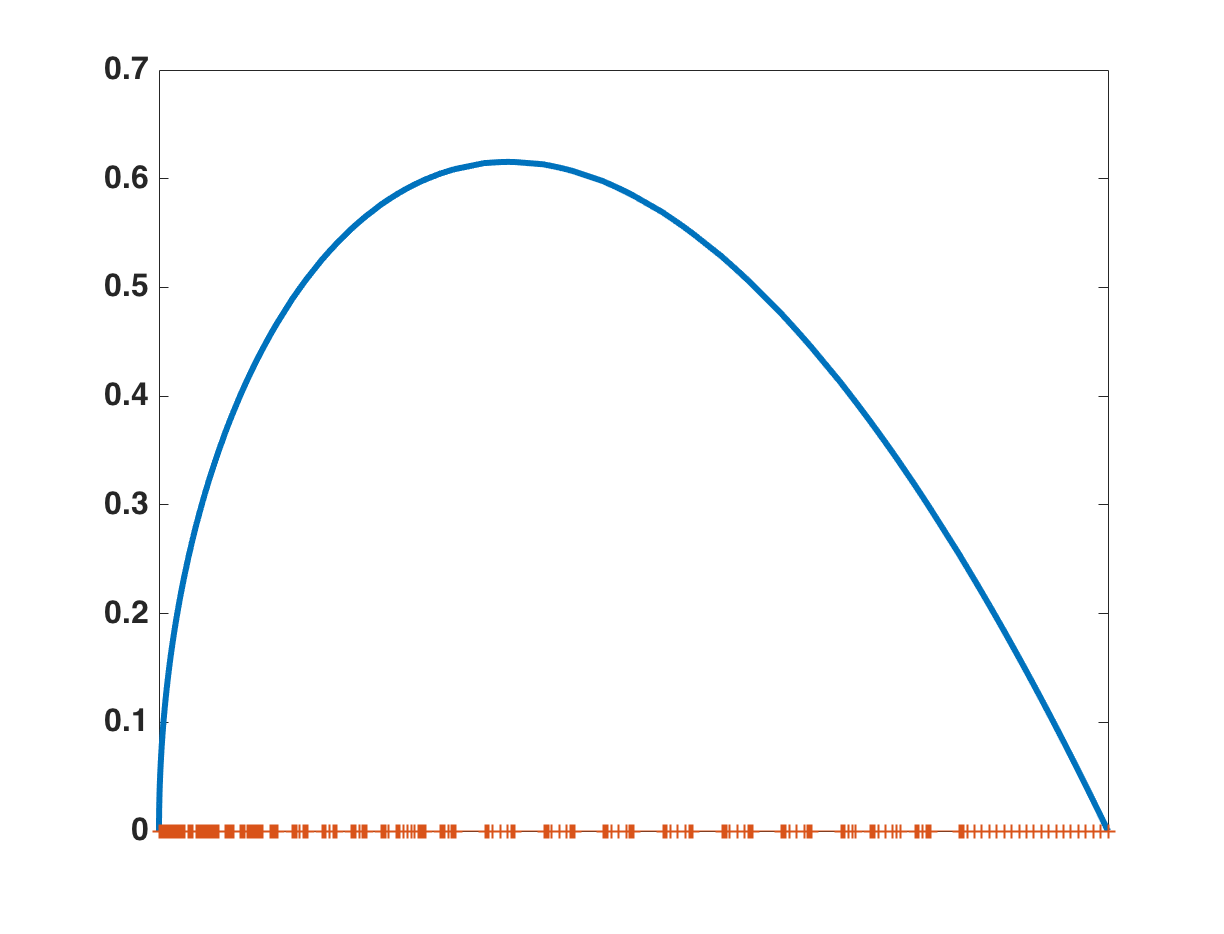}
\includegraphics[width=0.48\linewidth]{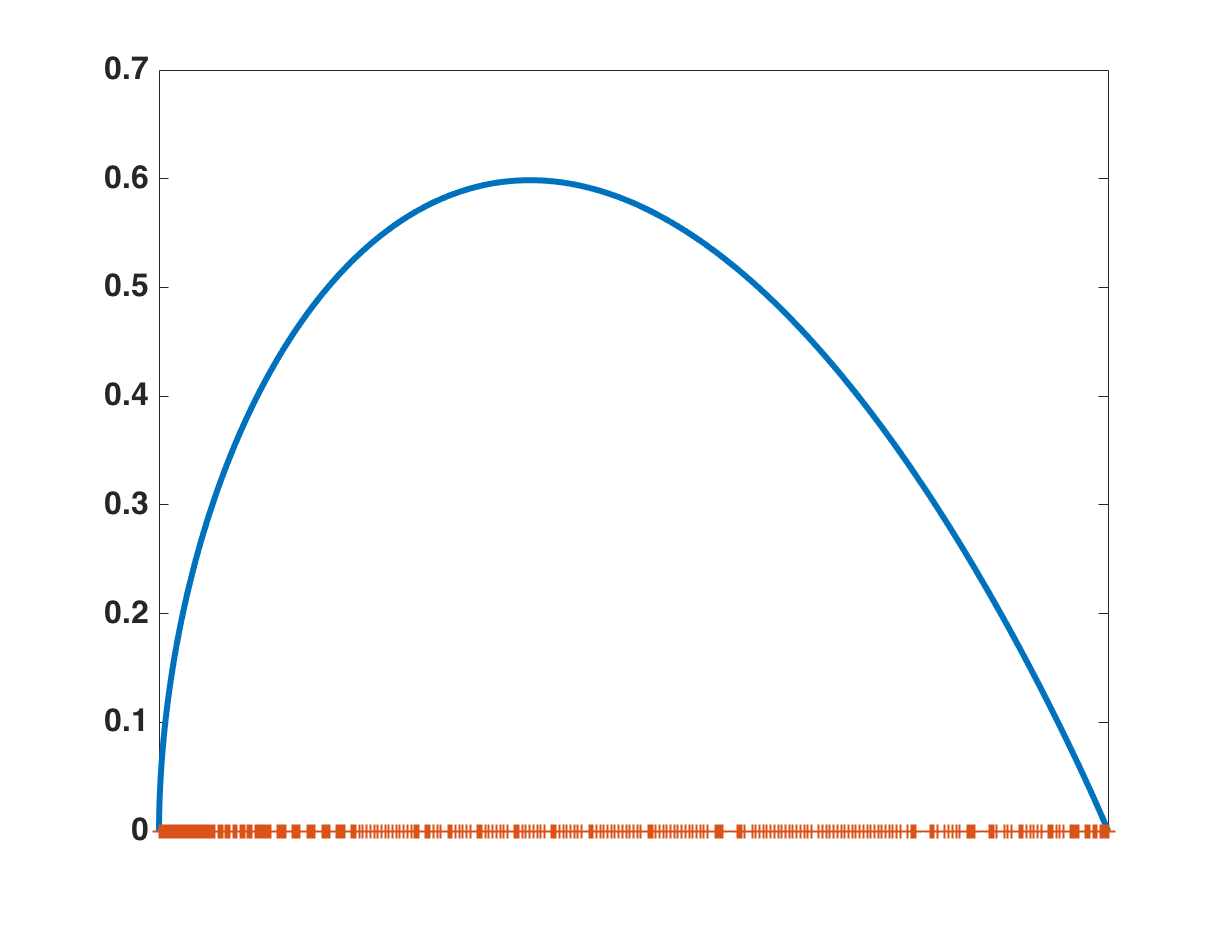}
\end{figure}

Figures \ref{fig:exp3-AFEM-L2} and \ref{fig:exp3-AFEM-Linfty} show the $L^2$ and $L^{\infty}$ errors of FEM on uniform grid and adaptive grids.  We can see that FEM on uniform grid suffers severe convergence degradation for both cases while AFEM recovers that the optimal convergence rate.   In Figure \ref{fig:exp3-AFEM-mesh}, we plot the numerical solution on adaptive meshes to show the refined grid at the left end point because in both cases $\kappa_1 \gg \kappa_2$.  This demonstrates the efficiency of our error estimator and adaptive algorithm.

\begin{example} \label{exp:4}
We consider the following fractional boundary value problem with nonhomogeneous boundary conditions and low order term.
\begin{align}
&\mathcal{D}_{x}^{\af}u(x) - \lambda^2 u=-(1+\sin(x)),\quad x\in(0,1),\\
&u(0)=0,\quad u(1)=1.
\end{align}
\end{example}
We use this example to demonstrate that our approach can be easily generalized to handle nonhomogeneous boundary conditions as well as a low order term.  In this numerical experiment, we choose $\alpha = 1.5$ and $\lambda = 0.5$.

\begin{figure}[ht!] \centering
\caption{Example \ref{exp:4}: Convergence rate of $L^2$ error (left) and $L^{\infty}$ error (right) comparisons between FEM on uniform grid and AFEM ($\alpha = 1.5$).} \label{fig:exp4-compare}
\includegraphics[width=0.48\linewidth]{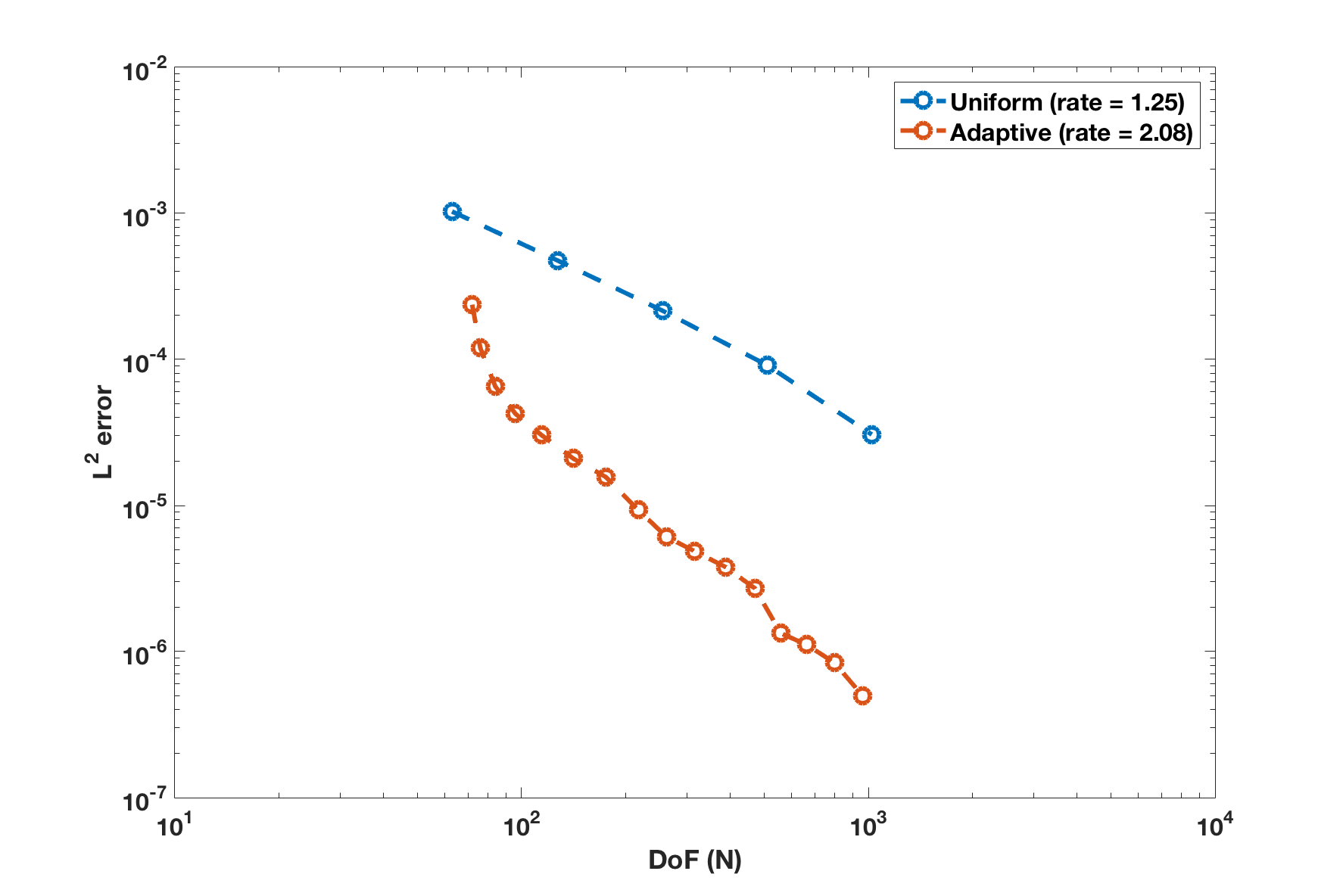}
\includegraphics[width=0.48\linewidth]{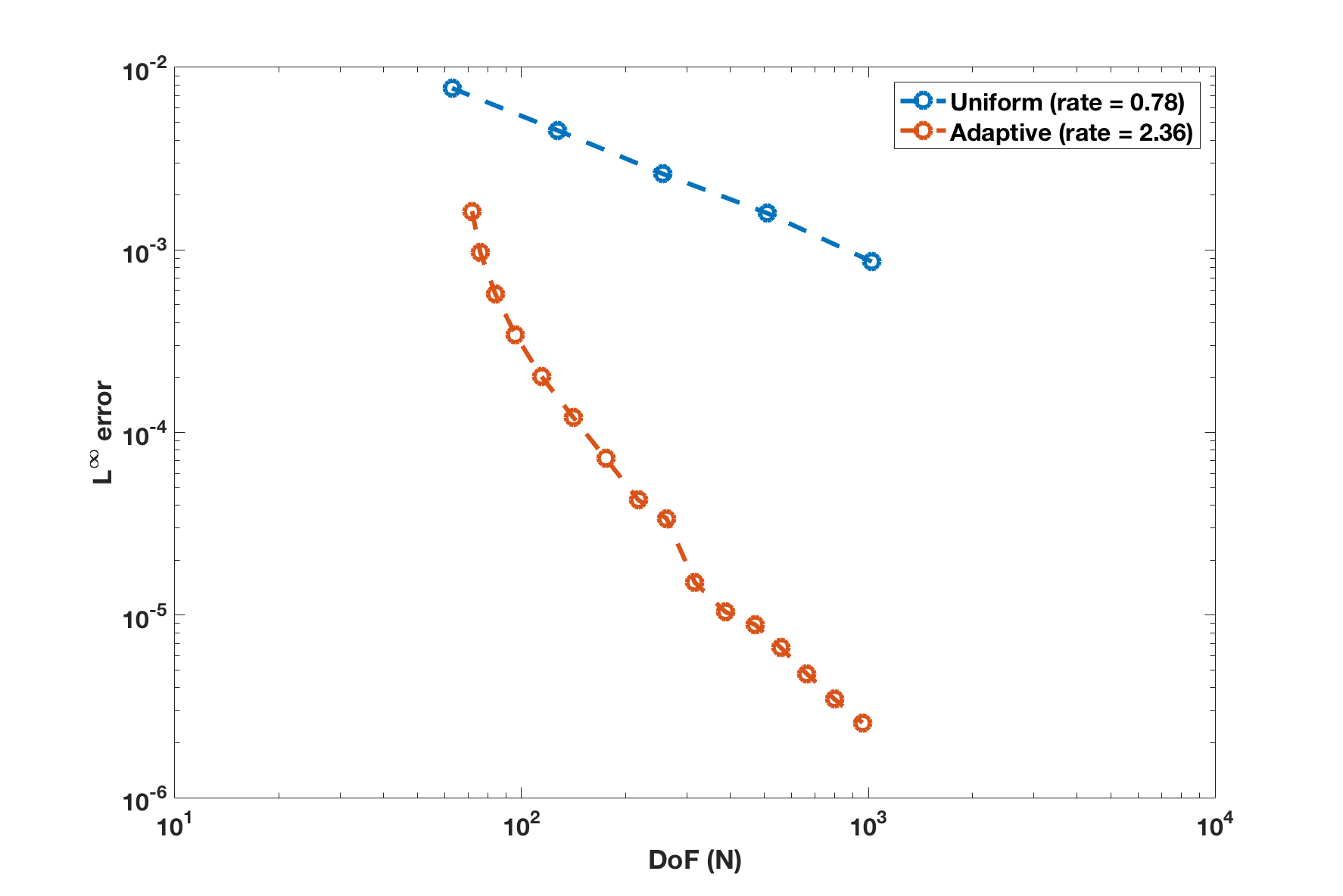}
\end{figure}

\begin{figure}
\centering
\caption{Example \ref{exp:4}: Numerical solution on adaptive mesh ($\alpha = 1.5$).} \label{fig:exp4-solution}
\includegraphics[width=0.7\linewidth]{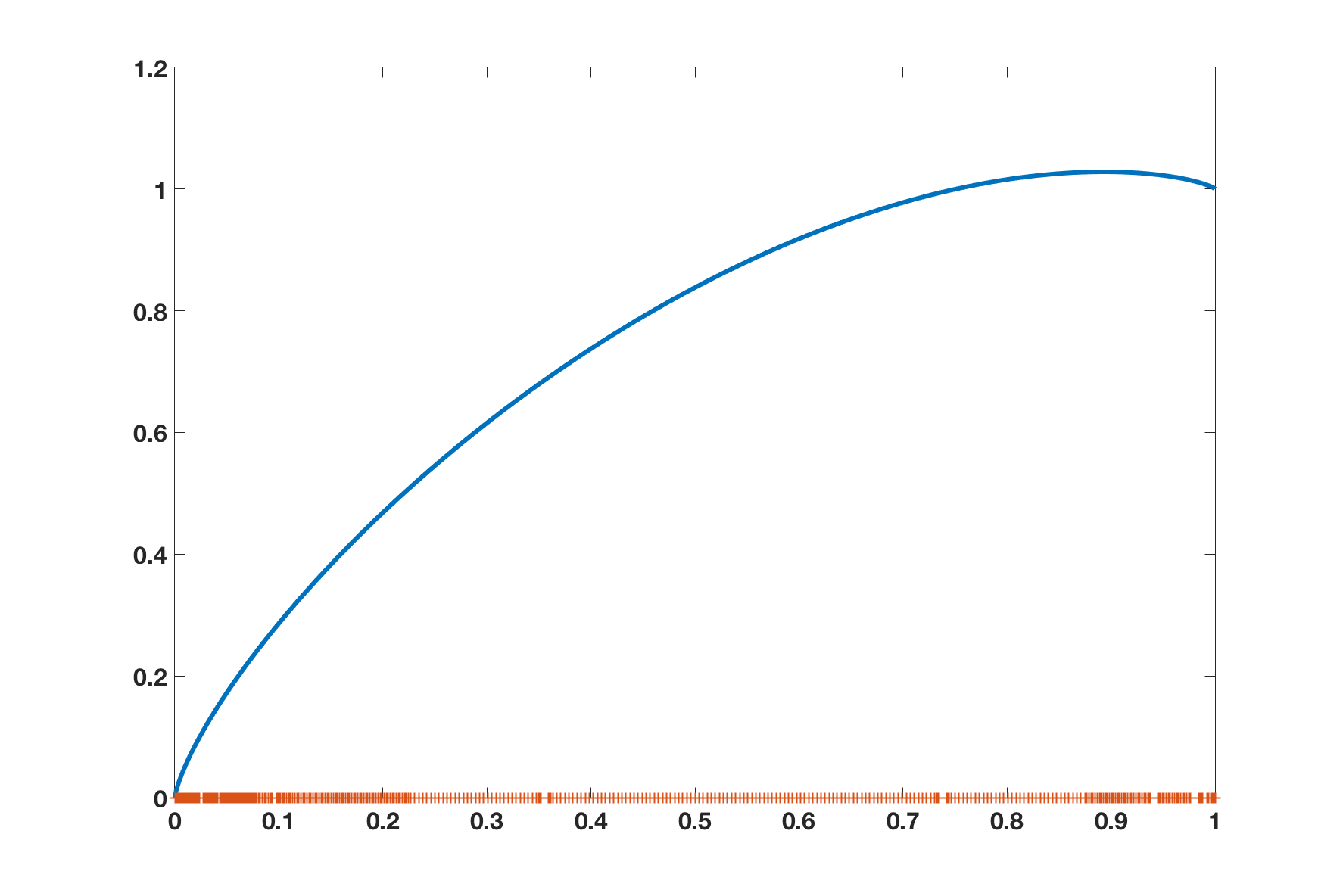}
\end{figure}

From Figure~\ref{fig:exp4-compare}, we can see that, on the uniform mesh, the convergence rates of both $L^2$ and $L^{\infty}$ errors are not optimal.  However, using the AFEM method, the optimal convergence rates of both $L^2$ and $L^{\infty}$ errors have been recovered, which demonstrates the effectiveness and robustness of the our AFEM methods.  We plot the numerical solutions on the adaptive mesh in Figure~\ref{fig:exp4-solution} and we can clearly see the adaptive refinement near the boundaries which demonstrates that our error estimates captures the singularities well, hence pointing to the overall robustness of our AFEM algorithm.

\section{Conclusion}

In this paper, we presented an adaptive FEM (AFEM) method for a fractional differential equation with Riesz derivative, targeting in particular non-smooth solutions that they may arise even in the presence of smooth right-hand-sides in the equation. To this end, uniform grids result in suboptimal and in fact sub-linear convergence rate, while the AFEM yields optimal second-order accuracy. The demonstrated efficiency of the method is based on combining two effective ideas, which act synergistically. First, we approximated the singular kernel in the fractional derivative using an H-matrix representation, and second, we employed a geometric multigrid method with linear overall computational complexity. In the current paper, we developed these ideas for the one-dimensional case but the greater challenge is to consider higher dimensions, where adaptive refinement has to resolve both solution singularities and geometric singularities around the boundaries.

\end{document}